\documentclass[aap,preprint]{imsart}

\RequirePackage{amsmath,amsthm,amsfonts,amssymb}
\RequirePackage[numbers]{natbib}
\RequirePackage[colorlinks,citecolor=blue,urlcolor=blue]{hyperref}

\RequirePackage{mathtools,relsize, nicefrac, etoolbox, xparse, upgreek, subcaption, enumerate}
\RequirePackage[english]{babel}
\RequirePackage[draft=false,babel,tracking=true,kerning=true,spacing=true]{microtype} 

\startlocaldefs

\numberwithin{equation}{section}
\swapnumbers
\theoremstyle{plain}
\newtheorem{theorem}{Theorem}[section]
\newtheorem{lemma}[theorem]{Lemma}
\newtheorem{corollary}[theorem]{Corollary}
\newtheorem{proposition}[theorem]{Proposition}
\newtheorem{assumption}[theorem]{Assumption}
\newtheorem{remark}[theorem]{Remark}
\theoremstyle{remark}
\newtheorem*{example}{Example}
\newtheorem{definition}[theorem]{Definition}
\newtheorem{claim}[theorem]{Claim}

\newcommand{\MTkillspecial}[1]{
	\begingroup%
	\catcode`\&=9%
	\let\\\relax%
	\scantokens{#1}%
	\endgroup%
}
\newcommand{\MTemptyplaceholder}{\:\cdot\:}
\DeclarePairedDelimiter\abs\lvert\rvert
\reDeclarePairedDelimiterInnerWrapper\abs{star}{%
	\mathopen{#1\vphantom{\MTkillspecial{#2}}\kern-\nulldelimiterspace\right.}%
	\ifblank{#2}{\MTemptyplaceholder}{#2}%
	\mathclose{\left.\kern-\nulldelimiterspace\vphantom{\MTkillspecial{#2}}#3}%
}
\DeclarePairedDelimiterXPP\iprodWrapper[3]{}{\langle}{\rangle}{#1}{
	\ifblank{#2}{\MTemptyplaceholder}{#2},
	\ifblank{#3}{\MTemptyplaceholder}{#3}
}
\NewDocumentCommand\iprod{ s o m m }{
	\IfBooleanTF {#1}
	{ \iprodWrapper*{\IfNoValueF{#2}{_{#2}}}{#3}{#4} }
	{ \iprodWrapper {\IfNoValueF{#2}{_{#2}}}{#3}{#4} }
}

\DeclarePairedDelimiterXPP\normWrapper[2]{}\lVert\rVert{#1}{\ifblank{#2}{\MTemptyplaceholder}{#2}}
\NewDocumentCommand\norm{ s o m }{
	\IfBooleanTF {#1}
	{ \normWrapper*{\IfNoValueF{#2}{_{#2}}}{#3}}
	{ \normWrapper {\IfNoValueF{#2}{_{#2}}}{#3}}
}

\providecommand\given{}

\DeclarePairedDelimiterX\Set[1]\{\}{%
	\renewcommand\given{\SetSymbol[\delimsize]}
	#1
}
\DeclarePairedDelimiterXPP\BorelSet[1]{\mathfrak{B}}(){}{%
	#1
}
\DeclarePairedDelimiterXPP\ACSet[1]{\ACfunc}(){}{%
	#1
}
\DeclarePairedDelimiterXPP\ProbWrapper[2]{#1}(){}{
	\renewcommand\given{\nonscript\:\delimsize\vert\nonscript\:\mathopen{}}
	#2
}

\NewDocumentCommand\prob{ s O{} O{} O{} m }{
	\ifblank {#5}{\mathrm{P}_{#2}^{#3}}
		{\IfBooleanTF {#1}
			{ \ProbWrapper*{\mathrm{P}_{#2}^{#3}}{#5} }
			{ \IfNoValueTF{#4}
				{
					\ProbWrapper{\mathrm{P}_{#2}^{#3}}{#5}
				}
				{
					\ProbWrapper[#4]{\mathrm{P}_{#2}^{#3}}{#5}
				}
			}
		}
}
\NewDocumentCommand\qrob{ s O{} O{} O{} m }{
	\ifblank {#5}{\mathrm{Q}_{#2}^{#3}}
		{\IfBooleanTF {#1}
			{ \ProbWrapper*{\mathrm{Q}_{#2}^{#3}}{#5} }
			{ \IfNoValueTF{#4}
				{
					\ProbWrapper{\mathrm{Q}_{#2}^{#3}}{#5}
				}
				{
					\ProbWrapper[#4]{\mathrm{Q}_{#2}^{#3}}{#5}
				}
			}
		}
}
\DeclarePairedDelimiterXPP\EVWrapper[2]{#1}[]{}{
	\renewcommand\given{\mathrel{}\mathclose{}\delimsize\vert\mathopen{}\mathrel{}}
	#2
}
\NewDocumentCommand\EV{ s O{} O{} O{} m }{
	\ifblank {#5}{\mathrm{E}_{#2}^{#3}}
		{\IfBooleanTF {#1}
			{ \EVWrapper*{\mathrm{E}_{#2}^{#3}}{#5} }
			{ \IfNoValueTF{#4}
				{
					\EVWrapper{\mathrm{E}_{#2}^{#3}}{#5}
				}
				{
					\EVWrapper[#4]{\mathrm{E}_{#2}^{#3}}{#5}
				}
			}
		}
}

\newcommand{\inv}{^{{\mathsmaller{-1}}}}

	\renewcommand{\phi}{\varphi}			
	\renewcommand{\epsilon}{\varepsilon}	
	\renewcommand{\theta}{\vartheta}		
	\renewcommand{\Delta}{\varDelta}		
	
	\newcommand{\identity}{\mathbf{I}}		

\DeclareMathOperator{\var}{Var}
\DeclareMathOperator{\cov}{Cov}
	
\newcommand{\D}{\ensuremath{\mathrm{d}}}
\renewcommand{\exp}[1]{\operatorname{exp}\left\{#1\right\}} 

		\newcommand{\smallo}{
	  \mathchoice
	    {{\scriptstyle\mathcal{O}}}
	    {{\scriptstyle\mathcal{O}}}
	    {{\scriptscriptstyle\mathcal{O}}}
	    {\scalebox{.7}{$\scriptscriptstyle\mathcal{O}$}}
	  }	

\endlocaldefs

\begin{document}

\begin{frontmatter}
\title{Estimation for the reaction term in semi-linear SPDEs\\
 under small diffusivity}
\runtitle{Small diffusivity: Estimation for the reaction term}
\thankstext[]{T1}{We thank Gregor Pasemann for many fruitful discussions. Financial support by Deutsche Forschungsgesellschaft (DFG) via IRTG 2544 and CRC 1294 is gratefully acknowledged.}

\begin{aug}
\author[A]{\fnms{Sascha} \snm{Gaudlitz}\ead[label=e1,mark]{sascha.gaudlitz@hu-berlin.de}}
\and
\author[A]{\fnms{Markus} \snm{Reiß}\ead[label=e2,mark]{mreiss@math.hu-berlin.de}}
\address[A]{Institut für Mathematik, Humboldt-Universität zu Berlin \printead{e1,e2}}

\end{aug}

\begin{abstract}

We consider the estimation of a non-linear reaction term in the stochastic heat or more generally in a semi-linear stochastic partial differential equation (SPDE). 
Consistent inference is achieved by studying a small diffusivity level, which is realistic in applications.
Our main result is a central limit theorem for the estimation error of a parametric estimator, from which confidence intervals can be constructed. Statistical efficiency is demonstrated by establishing local asymptotic normality. The estimation method is extended to local observations in time and space, which allows for non-parametric estimation of a reaction intensity varying in time and space. Furthermore, discrete observations in time and space can be handled. The statistical analysis requires advanced tools from stochastic analysis like Malliavin calculus for SPDEs, the infinite-dimensional Gaussian Poincar\'e inequality and regularity results for SPDEs in $L^p$-interpolation spaces.
\end{abstract}

\begin{keyword}[class=MSC]
\kwd[Primary ]{60H15}
\kwd{62F12}
\kwd{62G05}
\end{keyword}

\begin{keyword}
\kwd{Fractional heat equation}
\kwd{Poincaré Inequality}
\kwd{LAN property}
\kwd{Maximum likelihood estimation}
\kwd{Splitting trick}
\kwd{Non-parametric estimation}
\kwd{Discrete observations}
\end{keyword}

\end{frontmatter}


\section{Introduction}\label{sec:Intro}

\subsection{\textbf{Overview}}

While theory and also numerics for stochastic partial differential equations (SPDEs) have advanced significantly over the last years (see e.g.\ \cite{Hairer2009} and \cite{Lord2014} for a comprehensive account), the statistical methodology is less developed and some core  statistical questions are still unresolved. Statistics is not only vital for modeling with SPDEs in concrete applications, but poses also challenging mathematical problems for the analysis of SPDEs. Consider the following class of semi-linear SPDEs
	\begin{equation}
		\begin{cases}
			\D X_t = \nu A X_t\,\D t + \theta F(X_t)\,\D t +  B\D W_t,&t\in[0,T],\\
			X_0\quad\mathcal{F}_0\text{-measurable},
		\end{cases}\label{eq:SPDE}
	\end{equation}
on a filtered probability space $(\Omega,\mathcal{F},(\mathcal{F}_t)_{t\ge 0},\prob*{})$ supporting a cylindrical Brownian motion $(W_t)_{t\ge 0}$ with values in $\mathcal{H}\coloneqq L^2(\Lambda)$, where $\Lambda\subset{\mathbb R}^d$ is an open, bounded domain, $A\colon \operatorname{dom}(A)\subset \mathcal{H}\to \mathcal{H}$ is the generator of a semi-group $(S_t)_{t\ge 0}$ on $\mathcal{H}$, $F\colon \mathcal{H}\to\mathcal{H}$ is a continuous, possibly non-linear reaction term and $B\colon \mathcal{H}\to \mathcal{H}$ is a bounded linear dispersion operator. Our aim is to estimate the parameter $\theta\in\mathbb{R}$ for a small diffusivity level $\nu>0$. The terminology of \emph{diffusion} and \emph{reaction} term originates from the specific case of the stochastic reaction-diffusion or semi-linear heat equation with the Laplace operator $A=\Updelta$ in \eqref{eq:SPDE}.

Estimation of the diffusivity parameter $\nu$ has been extensively studied. Using the spectral decomposition of $A$,
\cite{Huebner1993, Huebner1995} have developed a general spectral  method for linear equations. It allows to estimate $\nu$ consistently from the dynamics of spectral statistics (Fourier modes), even from observations with a finite time horizon $T$. The spectral estimation method has subsequently been extended to cover also non-linear SPDEs like the stochastic Navier-Stokes equation \cite{Cialenco2011} and more general semi-linear equations \cite{Pasemann2020, Pasemann2021}. The non-parametric estimation of a space-dependent diffusivity $\nu$, based on local measurements only, has been pursued by \cite{Altmeyer2021, Altmeyer2020} and by \cite{Altmeyer2020b} with application to a concrete SPDE model for cell motility data. In the case $B=\sigma_t \identity$, that is driving space-time white noise of unknown time-dependent level $\sigma_t>0$, \cite{Bibinger2020} have used the realised volatility, based on high-frequency observations in time, to estimate $\sigma_t$. An estimator based on multipower variations in \cite{Chong2020} allows also to estimate a space-time dependent $\sigma$ from high-frequency observations. These approaches have been successfully extended to estimating $\nu$, $\sigma$ \cite{Markussen2003, Cialenco2020, Chong2020} and furthermore also $\theta F$ if, additionally, the time horizon $T$ increases to infinity \cite{Hildebrandt2021b}. To the best of our knowledge, the only estimation of the non-linear reaction intensity $\theta$ for a fixed time horizon $T$, has been conducted in the series of papers \cite{Ibragimov1999, Ibragimov2000, Ibragimov2001, Ibragimov2003} for the small-noise regime $\sigma\to 0$, where the limiting SPDE becomes deterministic. For a more thorough review of statistical methods for SPDEs we refer to Chapter 6 of \cite{Lototsky2017} and the excellent survey \cite{Cialenco2018}.

A consequence of the theory in \cite{Huebner1995} is that usually $\theta$ cannot be estimated consistently even from full space-time observations, when $T$, $\nu$ and $\sigma$ are fixed, e.g. in the case of the Laplace operator $A=\Updelta$ and $B=\sigma\identity$ in $d=1$.  Long time asymptotics $T\to\infty$ require ergodicity properties of the solution process $X$, which are difficult to verify and furthermore not satisfied for several models of interest. Yet, parabolic SPDE models in applications often focus on lower order transport or reaction terms which describe the fundamental dynamics and become superposed with diffusion phenomena and random fluctuations. In concrete examples, the diffusivity $\nu$ will mostly be of the order of the reaction term \cite{Soh2010} or  even smaller \cite{Alonso2018, Flemming2020, Altmeyer2020b}, so that the reaction term has a clear impact on the dynamics. From a statistician's point of view this indicates that inference on the reaction parameters must be feasible.

We corroborate this view by constructing an estimator $\hat{\theta}_\nu$ and by showing that its estimation error for fixed $T$ and $B$ is small for small diffusivity $\nu$. More precisely, we derive a central limit theorem (CLT) for the scaled estimation error $\phi(\nu)^{1/2}(\hat{\theta}_\nu-\theta)$ as $\nu\to 0$ with a convergence rate $\phi(\nu)^{-1/2}$ depending on the generator $A$, e.g.\ $\phi(\nu)^{-1/2}\sim\nu^{1/4}$ for the one-dimensional Laplace operator $\Updelta$. Consequently, we can quantify the uncertainty of the estimator $\hat\theta_\nu$ by explicit and data-driven confidence intervals. The rate of converge is determined by the growth of the spatial $L^2$-norm of $X$ as $\nu \to 0$. Moreover, we are able to establish the LAN property (local asymptotic normality) of our model with the same convergence rate and variance, showing that our estimator is statistically efficient. The estimator is based on the maximum-likelihood method and uses the full spatio-temporal observation $(X_t(y),t\in[0,T],y\in\Lambda)$. Simulation results confirm the theoretical findings and exhibit good finite sample behaviour.

We are able to extend the results to less informative observation schemes. First, we consider observations of $X$ in a small spatio-temporal domain and show that estimation is still feasible. The convergence rate scales with the size of the spatio-temporal domain tending to zero. These local observations then even  permit to estimate a space-time varying parameter $\theta(y,t)$ non-parametrically. We find standard pointwise estimation rates over anisotropic H\"older classes. Finally, we treat discrete observations in space and time. We propose a discretised estimator and derive sufficient conditions on the discretization mesh to ensure the same asymptotics as for $\hat\theta_\nu$.

While the derivation of the maximum-likelihood estimator (MLE) relies on a proper analysis of the Girsanov theorem for SPDEs, the analysis of the estimation error is non-trivial. We establish the CLT by martingale methods, for which we have to control the fluctuations of the observed Fisher information around its expectation. To this end, we employ Malliavin calculus for SPDEs and the infinite-dimensional Poincar\'e inequality on one hand and the Da Prato-Debussche splitting trick to control the non-linear part of the semi-linear SPDE on the other hand. The local observation scheme requires uniform control of certain statistics, as functions of time and space, which is achieved by corresponding  properties of the semi-group and its Green function as well as by Malliavin calculus again. The handling of discrete observations calls for regularity results in $L^p$-interpolation and $L^p$-Sobolev spaces to give a tight control of the discretization error as $\nu\to 0$. These have not yet been explored in the literature and might be of independent interest.

In Subsection \ref{subsec:Setting} we present the exact setting and assumptions as well as the fractional heat equation, which serves as our guiding example in the sequel. Section \ref{sec:MainResults} constructs the MLE $\hat{\theta}_\nu$, presents its asymptotics and discusses its efficiency. The practical implementation and numerical examples are discussed in Subsection \ref{subsec:Numerics}. We will  explain the main steps for proving the CLT in Subsection \ref{subsec:towardsCLT}. Reduced observation schemes are treated in Section \ref{sec:discretize}, allowing for parametric estimation from local observations, non-parametric estimation via localisation and parametric estimation from discrete observations. Most proofs are postponed to Appendices \ref{subsec:Proof_Lemma_Example}--\ref{sec:TechnicalTool}.

\subsection{\bf{Setting and guiding example}}\label{subsec:Setting}

	For functions $u,v\colon (0,\infty)\to\mathbb{R}$ we write $u\lesssim v$ (equivalently $v\gtrsim u$) if $\limsup_{\nu\to 0}u(\nu)/v(\nu)<\infty$ and $u\sim v$ if $u\lesssim v$ and $u\gtrsim v$.
	
	Let $\Lambda\subset\mathbb{R}^d$ be an open and bounded domain with Lipschitz boundary. Define $\mathcal{H}\coloneqq L^2(\Lambda)$ with inner product $\iprod*{}{}$ and assume that $A\colon \operatorname{dom}(A)\subset\mathcal{H}\to\mathcal{H}$ is a densely defined operator generating an analytic contraction semi-group $(S_t)_{t\ge 0}$ on $\mathcal{H}$. Then the semi-group generated by $\nu A$ is given by $(S_{\nu t})_{t\ge 0}$. We denote by $\norm*{}_{\operatorname{HS}}$ the Hilbert-Schmidt norm and by $\norm*{}$ the spectral (operator) norm of operators on $\mathcal{H}$. $F\colon\mathcal{H}\to\mathcal{H}$ is a, possibly non-linear, operator and $B\colon\mathcal{H}\to\mathcal{H}$ is a bounded linear operator with bounded inverse $B\inv$. For a parameter $\theta\in\mathbb{R}$ we consider the semi-linear SPDE \eqref{eq:SPDE}, assuming $\EV{e^{\norm{X_0}^2}}<\infty$ for the initial condition.  In Subsection \ref{subsec:Nonparametric} we consider a space-time dependent parameter $\theta(y,t)$, given by a locally Hölder-continuous function on $\Lambda\times [0,T]$. In this case we understand $[\theta_t F(z)](y)=\theta(y,t)[F(z)](y)$, $z\in L^2(\Lambda)$, $(y,t)\in\Lambda\times[0,T]$.
	
	Besides these standing requirements, we proceed to state a collection of assumptions that will be used in the following. We will state for every result individually, which of these assumptions are necessary. Lemma \ref{lem:examples_Assumptions} (below), proven in Appendix \ref{subsec:Proof_Lemma_Example}, shows that one important class of SPDEs for which the following Assumptions hold is given by the following example. It will be considered throughout the presentation.
	
	\begin{example}[Fractional heat equation]
	Fix some $l>0$, let $\Lambda=(0,l)$ and consider the semi-linear SPDE
	\begin{equation}
		\begin{cases}
			\D X_t(y) = -\nu(-\Updelta)^{\alpha/2}X_t(y)\,\D t + \theta f(X_t(y))\,\D t + \sigma(y)\D W_t(y),& t\in[0,T],\\
			X_0\equiv 0,&
		\end{cases}\label{eq:Guiding_Example}
	\end{equation}
	for $\theta\in\mathbb{R}$. $(-\Updelta)^{\alpha/2}$, $\alpha\in(1,2]$, is the (spectral) fractional Laplacian with Dirichlet boundary conditions, defined by
	\begin{equation*}
		(-\Updelta)^{\alpha/2}z \coloneqq \sum_{k=1}^\infty\lambda_k^{\alpha/2}\iprod*{z}{e_k}e_k
	\end{equation*}
	for any $z\in L^2(\Lambda)$ providing convergence. $(\lambda_k,e_k)_{k\in\mathbb{N}}$ is the singular value decomposition of $-\Updelta$ on $\Lambda$, i.e.\ $e_k(x) = \sqrt{2/l}\sin(k\pi x/l)$ and $\lambda_k= (k\pi)^2/l^2$. $B$ is a multiplication operator $Bz = \sigma z$, $z\in L^2(\Lambda)$, where $\sigma\colon \Lambda\to\mathbb{R}$ is smooth, bounded and $\inf_\Lambda\sigma>0$. Additionally, $F(z)=f\circ z$, $z\in L^2(\Lambda)$, where $f\colon\mathbb{R}\to\mathbb{R}$  has a bounded, Lipschitz-continuous derivative.
\end{example}

We postulate well-posedness of \eqref{eq:SPDE}:

\begin{assumption}\label{assump:general}
		The semi-linear SPDE \eqref{eq:SPDE} has a unique weak solution $X=(X_t)_{t\in[0,T]}$ with paths in $C([0,T],L^2(\Lambda))$, satisfying the variation-of-constants formula
		\begin{equation*}
			X_t = S_{\nu t}X_0 + \int_0^t S_{\nu(t-s)}\theta F(X_s)\,\D s + \int_0^t S_{\nu(t-s)}B\,\D W_s.
		\end{equation*}
\end{assumption}

	 We refer to Theorems 7.2 and 7.5 of \cite{DaPrato2014} for general conditions under which Assumption \ref{assump:general} is satisfied.
	
\begin{definition}
	For $\nu>0$ and $t\in[0,T]$ we set $\phi(\nu,t)\coloneqq \int_0^t \norm*{S_{\nu s}}_{\operatorname{HS}}^2\,\D s$ and $\phi(\nu):=\phi(\nu,T)$, which will later describe the convergence rate of the estimator.
\end{definition}

	\begin{remark}\label{rmk:beta_growth}
$\phi(\nu,t)$ is increasing in $t$ and $\phi(\nu,T)<\infty$ holds for all $\nu>0$ under Assumption \ref{assump:general}. Because of $\norm*{S_{\nu t}}_{\operatorname{HS}}\to\infty$ for $\nu\to 0$, we have $\phi(\nu,t)\to\infty$ as $\nu\to 0$. As $(S_t)_{t\ge 0}$ is a contraction semi-group, $t\mapsto\norm*{S_t}_{\operatorname{HS}}$ is decreasing. This gives for all $0<t_1\le t_2\le T$, $\nu>0$, the two-sided bound
		\begin{equation}
			\phi(\nu,t_1)\le \phi(\nu,t_2)\le (t_2/t_1)\phi(\nu,t_1),\label{eq:compatibility}
		\end{equation}
		which implies in particular $\phi(\nu,t)\thicksim \phi(\nu)$ for fixed $t>0$.


\end{remark}

	 As we shall use Malliavin calculus to control certain statistics, we denote by $\mathbb{H}$ the Hilbert space $L^2(\Lambda)$ with the weighted inner product $\iprod*{}{}_{\mathbb{H}}\coloneqq\iprod*{B^\ast\MTemptyplaceholder}{B^\ast\MTemptyplaceholder}$. We can interpret $(B\D W_t)_{t\in[0,T]}$ as an isonormal process $\mathcal{W}$ on $\mathfrak{H}\coloneqq L^2([0,T],\mathbb{H})$. The Malliavin derivative $\mathcal{D}Z$ of a Malliavin differentiable random variable $Z$ based on $\mathcal{W}$ belongs to $\mathfrak{H}$ and we use the notation $\mathcal{D}Z = (\mathcal{D}_\tau Z)_{\tau\in[0,T]}$ with $\mathcal{D}_\tau Z\in\mathbb{H}$.

The following assumption paves the ground for connecting $\int_0^T S_{\nu(t-s)}B\,\D W_s$ to the Malliavin calculus. The condition for $A$ is satisfied e.g.\ if $A$ is the generator of a Markov process with transition density $P_t(x,y)$, in which case $G_t(x,y) = P_t(x,y)$.
	\begin{assumption}~\label{assump:generalA}
		The semi-group $S_t$ takes the form $S_t z(y)=\iprod*{G_t(y,\MTemptyplaceholder)}{z}$, $z\in L^2(\Lambda)$, $y\in\Lambda$, $t\in (0,T]$, with kernel (Green function) $G_{\MTemptyplaceholder}(y,\MTemptyplaceholder)\in\mathfrak{H}$ for all $y\in \Lambda$. The operator $B$ takes the multiplicative form $Bz = \sigma z$, $z\in L^2(\Lambda)$, for some bounded $\sigma\colon\mathbb{R}\to\mathbb{R}_+$ such that $1/\sigma$ is bounded.
\end{assumption}

\begin{remark}\label{rmk:Skorokhod_Intergal}
Define the Skorokhod integral
		\begin{equation*}
			\int_0^t\int_\Lambda G_{\nu(t-s)}(y,\eta)\mathcal{W}(\D s,\D\eta)\coloneqq \delta(G_{\nu(t-\MTemptyplaceholder)}(y,\MTemptyplaceholder)),
		\end{equation*}
		where $\delta$ is the divergence operator of Malliavin calculus, see e.g.\ \cite{Nualart2009}. Then
		\begin{equation*}
			\int_0^t S_{\nu(t-s)}B\,\D W_s = \int_0^t\int_\Lambda G_{\nu(t-s)}(\MTemptyplaceholder,\eta)\mathcal{W}(\D s,\D\eta)
		\end{equation*}
		holds with equality in $L^2(\Omega,C([0,T],L^2(\Lambda)))$, see e.g.\ Propositions 4.3 and 4.4 of \cite{SanzSole2013} for a comparison of Da Prato/Zabczyk-type stochastic integrals and Skorokhod integrals.
\end{remark}

In order to control fluctuations of $X_t$ and to apply later the Poincaré inequality, we require standard properties of the Malliavin derivative of the point evaluation $X_t(y)$.
	\begin{assumption}\label{assump:specific_F}
 		$F$ is a Nemytskii-type operator in the sense that there exists a function $f\in C^1(\mathbb{R})$ with $\norm*{f'}_\infty\coloneqq\sup_{\Lambda}\abs*{f'}<\infty$ such that $F(z)=f\circ z$ for all $z\in L^2(\Lambda)$. $G_t$, $f$ and the initial condition $X_0$ are such that the unique weak solution $X$ to the SPDE \eqref{eq:SPDE} allows for a random field representation
		\begin{equation*}
			X_t(y) = S_{\nu t}X_0(y) + \int_0^t \int_\Lambda G_{\nu(t-s)}(y,\eta)f(X_s(\eta))\,\D\eta\D s + \int_0^t\int_\Lambda G_{\nu(t-s)}(y,\eta)\mathcal{W}(\D s,\D\eta)
		\end{equation*}
		with equality in $L^2(\Omega,C([0,T],L^2(\Lambda)))$. Furthermore, for all $(y,t)\in \Lambda\times[0,T]$, the Malliavin derivative $\mathcal{D}X_t(y)\in\mathfrak{H}$ of $X_t(y)$ exists and satisfies
		\begin{equation*}
			\mathcal{D}_\tau X_t(y) = G_{\nu(t-\tau)}(y,\MTemptyplaceholder)+\int_0^t \int_\Lambda G_{\nu(t-s)}(y,\eta)f'(X_s(\eta))\mathcal{D}_\tau X_s(\eta)\,\D\eta\D s
		\end{equation*}
		for $\tau\in [0,t)$ and $\mathcal{D}_\tau X_t(y)=0$ for $\tau\in [t,T]$.
\end{assumption}
Assumption \ref{assump:specific_F} implies Assumption \ref{assump:general}. There are many cases known in the literature where Assumption \ref{assump:specific_F} holds, compare \cite{SanzSole2005}.

\begin{lemma}\label{lem:examples_Assumptions}
	The fractional heat equation  \eqref{eq:Guiding_Example} satisfies Assumptions \ref{assump:generalA}, \ref{assump:specific_F} and therefore also Assumption \ref{assump:general}. We have
	\begin{equation*}	
		\phi(\nu,t) \sim \nu^{-1/\alpha}t^{1-1/\alpha},\quad \nu>0,\,t\in[0,T].
	\end{equation*}
\end{lemma}

In order to quantify the estimation error, we will impose a two-sided growth bound on $F$.
	\begin{assumption}\label{assump:general_F}
	There exist constants $0<d\le D<\infty$ such that
	\begin{equation*}
		d\norm*{z} \le \norm*{F(z)}\le D(\norm*{z} + 1),\quad z\in L^2(\Lambda).
	\end{equation*}
	\end{assumption}
	\begin{example}[continued]
		In the case of the fractional heat equation  \eqref{eq:Guiding_Example}, Assumption \ref{assump:general_F} is satisfied if we additionally assume  $\abs{f(y)}\ge d\abs*{y}$ for all $y\in\mathbb{R}$.
	\end{example}
	
\section{Main results}\label{sec:MainResults}

	Let us state our main results starting with the derivation of the MLE $\hat{\theta}_\nu$ for $\theta$ in \eqref{eq:SPDE}. The MLE satisfies a CLT which allows to construct confidence intervals. Moreover, the MLE is shown to be an asymptotically efficient estimator. A numerical example illustrates the finite sample performance and discusses further issues. The general strategies of proof are exposed, while the technical proof details are postponed to Appendix \ref{sec:Proofs_Main}.

\subsection{\textbf{Construction of the estimator and the CLT}}

	The  Girsanov Theorem allows  to construct a maximum-likelihood estimator $\hat{\theta}_\nu$. To this end, denote by $\prob*[\theta]{}$ the law of the solution $(X_t)_{t\in[0,T]}$ of the SPDE \eqref{eq:SPDE} on the path space $C([0,T],\mathcal{H})$ with its Borel $\sigma$-algebra.
	\begin{lemma}\label{lem:MLE_decomp}
		Grant Assumption \ref{assump:general}. The MLE $\hat{\theta}_\nu$ for $\theta$ in the semi-linear SPDE \eqref{eq:SPDE} exists and is given by
		\begin{equation*}
			\hat{\theta}_\nu\coloneqq \frac{\int_0^T \iprod*{B\inv F(X_t)}{B\inv[\D X_t - \nu AX_t\,\D t]}}{\int_0^T \norm*{B\inv F(X_t)}^2\,\D t},
		\end{equation*}
		provided the denominator does not vanish. We understand the numerator in the sense of the $L^2(\prob*[0]{})$-convergent series
		\begin{align*}
			\int_0^T \iprod*{B\inv F(X_t)}{B\inv[\D X_t - \nu AX_t\,\D t]}\,\D t \coloneqq\hspace{-15em}&\\
				&\sum_{k=1}^\infty \left[ \int_0^T \iprod*{B\inv F(X_t)}{u_k}\D\iprod*{u_k}{B\inv X_t}-\nu\int_0^T \iprod*{B\inv F(X_t)}{u_k}\iprod*{X_t}{A^\ast B\inv u_k}\,\D t\right]
		\end{align*}
		for any complete orthonormal system $(u_k)_{k\in\mathbb{N}}\subset\operatorname{dom}(A^\ast B\inv)$ of $\mathcal{H}$.
		
		$\hat{\theta}_\nu$ satisfies under $\prob*[\theta]{}$ the representation
		\begin{equation}
			\hat{\theta}_\nu = \theta + \frac{\mathcal{M}_\nu}{\mathcal{I}_\nu},\label{eq:MLE_decmop}
		\end{equation}
		with $\mathcal{M}_\nu\coloneqq \int_0^T\iprod*{B\inv F(X_t)}{\D W_t}$ and $\mathcal{I}_\nu\coloneqq\int_0^T \norm*{B\inv F(X_t)}^2\,\D t$.
	\end{lemma}
	\begin{remark}~
		\begin{enumerate}[(i)]
			\item We must analyse the properties of the MLE $\hat{\theta}_\nu$ under the data-generating law $\prob*[\theta]{}$. As the Girsanov Theorem ensures $\prob*[\theta]{}\ll\prob*[0]{}$, the series representation of the numerator also converges in $\prob*[\theta]{}$-probability, so that the MLE is well defined. The value of $\theta$ is fixed throughout the work and we will often omit the index $\theta$.
			\item The term $\mathcal{I}_\nu$ is commonly called observed Fisher information, noting that $\EV*{\mathcal{I}_\nu}$ is the classical Fisher information. Note already that $\mathcal{I}_\nu$ is the quadratic variation of the martingale $\mathcal{M}_\nu$ in $T$.
			\item The proof of Lemma \ref{lem:MLE_decomp} also shows that the log-likelihood $\ell$ under $\prob*[\theta]{}$ satisfies
			\begin{equation*}
				\ell(X,\theta)\overset{d}{=} \int_0^T \theta \iprod*{B\inv F(X_t)}{\D W_t} + \frac{1}{2}\int_0^T\theta^2 \norm*{B\inv F(X_t)}^2\,\D t.
			\end{equation*}
		\end{enumerate}
	\end{remark}
	
	The following CLT provides the main result on the estimation error of the MLE $\hat{\theta}_\nu$.
	\begin{theorem}\label{thm:CLTnon-linear}
	Grant Assumptions \ref{assump:generalA} and \ref{assump:specific_F}. Then
	\begin{equation*}
		\EV*{\mathcal{I}_\nu}^{1/2}(\hat{\theta}_\nu-\theta)\xrightarrow{d}N(0,1),\quad \nu\to 0,
	\end{equation*}
	holds, provided that $\EV*{\mathcal{I}_\nu}\to\infty$. Under Assumption \ref{assump:general_F} we have $\EV*{\mathcal{I}_\nu}\sim\phi(\nu)$, which tends to infinity.
\end{theorem}

This CLT paves the way for the construction of asymptotic confidence intervals.

	\begin{corollary}\label{cor:Confidence}
		Grant Assumptions \ref{assump:generalA} and \ref{assump:specific_F}. If $\EV*{\mathcal{I}_\nu}\to\infty$, then
		\begin{equation*}
			\mathcal{A}_{1-\bar{\alpha}}\coloneqq \left[\hat{\theta}_\nu-q_{1-\bar{\alpha}/2}\mathcal{I}_\nu^{-1/2},\hat{\theta}_\nu+q_{1-\bar{\alpha}/2}\mathcal{I}_\nu^{-1/2}\right],\quad \bar{\alpha}\in(0,1),
		\end{equation*}
		with the standard Gaussian $(1-\bar{\alpha}/2)$-quantile $q_{1-\bar{\alpha}/2}$ are asymptotic $(1-\bar{\alpha})$-confidence intervals of $\theta$ as $\nu\to 0$.
	\end{corollary}
	The proofs of Theorem \ref{thm:CLTnon-linear} and Corollary \ref{cor:Confidence} are explained in Section \ref{subsec:ProofCLT}. Details are postponed to Section \ref{sec:Proofs_CLT}.
	\begin{remark}~
		\begin{enumerate}[(i)]
			\item The  convergence rate $\phi(\nu)^{-1/2}=(\int_0^T \norm*{S_{\nu t}}_{\operatorname{HS}}^2\,\D t)^{-1/2}$ in an ${\mathcal O}_P$-sense holds already under less assumptions, see Propositions \ref{thm:non-linearrate} and \ref{lem:general_rate} below.
			\item If $F$ is linear, $X_0=0$ and $\sigma\equiv\sigma_0>0$, then the asymptotic variance of $\hat{\theta}_\nu$ does not depend on the noise level $\sigma_0$. More precisely, we have
			\begin{equation*}
				\EV*{\mathcal{I}_\nu} = \int_0^T \EV*{\norm*{\sigma_0^{-1} X_t}^2}\,\D t =  \int_0^T \EV*{\norm*{Z_t}^2}\,\D t,
			\end{equation*}
			where $Z$ solves the same (linear) SPDE as $X$ with $\sigma_0=1$.
			\item While the observation laws $\prob*[\theta_1]{}$ and $\prob*[\theta_2]{}$ for fixed $\nu>0$ and $\theta_1\neq\theta_2$ are equivalent, we can infer from Theorem \ref{thm:CLTnon-linear} that these laws separate as $\nu\to 0$ (they are not {\it contiguous}). So, even on a fixed time horizon $T$ we can identify $\theta$ in the limit $\nu\to 0$ without referring to the (formal) singular limiting SPDE
			\begin{equation*}
				\D X_t = \theta F(X_t)\,\D t +B \D W_t.
			\end{equation*}
		\end{enumerate}
	\end{remark}
	\begin{example}[continued]
		Invoking Lemma \ref{lem:examples_Assumptions},  Theorem \ref{thm:CLTnon-linear} holds in the case of the fractional heat equation  \eqref{eq:Guiding_Example}. Under the condition $\abs*{f(x)}\ge d\abs*{x}$, the rate of convergence is given by $\phi(\nu)^{-1/2}\sim \nu^{1/(2\alpha)}$. The order of the generator plays a crucial role: Smaller values of $\alpha$ induce less regularisation of the semigroup and result in faster convergence rates. This accords well with the result of \cite{Huebner1995}, who proved that in the spectral approach for linear SPDEs the convergence is the faster, the smaller the leading order. For the Laplacian ($\alpha=2$) we arrive at the rate $\nu^{1/4}$.
		
		Further intuition can be gained by considering the diagonalizable case where $F$ and $B$ are the identity and $X_0=0$.
We obtain
			\begin{equation*}
				\EV*{\mathcal{I}_\nu} = \int_0^T \int_0^t e^{2\theta s}\norm*{S_{\nu s}}_{\operatorname{HS}}^2\,\D s\D t\sim \nu^{-1/\alpha}\int_0^T \int_0^t e^{2\theta s}s^{-1/\alpha}\,\D s\D t.
			\end{equation*}
			This yields the rate
			\begin{equation*}
				\EV*{\mathcal{I}_\nu}^{-1/2}\sim\begin{cases}\nu^{1/(2\alpha)}e^{2\theta T}T^{1/(2\alpha)-1},&\theta>0,\\\nu^{1/(2\alpha)}T^{1/(2\alpha)-1},&\theta = 0,\\\nu^{1/(2\alpha)}T^{-1/2},& \theta<0,\end{cases}
			\end{equation*}
			as $\nu\to 0$, $T\to\infty$. This change of asymptotic behaviour in $T$ is well-known for the (real-valued) Ornstein-Uhlenbeck process (Proposition 3.46 of \cite{Kutoyants2013}).
			
			The $k$-th Fourier mode $X_t^k \coloneqq \iprod*{X_t}{e_k}$ satisfies the Ornstein-Uhlenbeck dynamics $\D X_t^k  = (-\nu\lambda_k + \theta)X_t^k\D t + \D W_t^k$, where $(W_t^k)_{k\ge 1}$ are independent Brownian motions. When we observe the first $K_\nu=\max\{k\ge 1\,|\,\nu\lambda_k\le 1\}$ Fourier modes, the drift factor $-\nu\lambda_k +\theta$ remains bounded. The corresponding MLE $\hat{\theta}_{K_\nu}$ for $\theta$ satisfies
			\begin{equation*}
				\hat{\theta}_{K_\nu}  = \frac{\sum_{k=1}^{K_\nu}\int_0^T X_t^k (\D X_t^k+\nu\lambda_kX_t^k\D t)}{\sum_{k=1}^{K_\nu}\int_0^T (X_t^k)^2\,\D t} =\theta + \frac{\frac{1}{K_\nu}\sum_{k=1}^{K_\nu}\int_0^T X_t^k \D W_t^k}{\frac{1}{K_\nu}\sum_{k=1}^{K_\nu}\int_0^T (X_t^k)^2\,\D t}.
			\end{equation*}
			Using the independence of the summands, we find $\hat{\theta}_{K_\nu}-\theta=\mathcal{O}_{\prob*{}}(K_\nu^{-1/2})=\mathcal{O}_{\prob*{}}(\nu^{-1/(2\alpha)})$.
	\end{example}
	
	\subsection{\textbf{Efficiency}}\label{subsec:Efficiency}
			
			It is in general not clear that the maximum-likelihood method yields an efficient estimator in terms of convergence rate and asymptotic variance. Here, it is indeed the case and follows from the local asymptotic normality (LAN) property (\cite{LeCam1960}, see also Chapter 2 of \cite{Ibragimov1981}).

			\begin{proposition}\label{lem:LANProperty}
				Grant Assumptions \ref{assump:generalA} and \ref{assump:specific_F}. For $h\in\mathbb{R}$ define $h_\nu\coloneqq h\EV[\theta]{\mathcal{I}_\nu}^{-1/2}$. If $\EV[\theta]{\mathcal{I}_\nu}\to\infty$ as $\nu\to 0$, then for any $\theta\in\mathbb{R}$  the log-likelihood  satisfies under $\prob*[\theta]{}$
				\begin{equation*}
					\ell(X,\theta+h_\nu)-\ell(X,\theta)\xrightarrow{d}N(-h/2,h^2),\quad \nu\to 0.
				\end{equation*}
				In particular, local asymptotic normality holds.
			\end{proposition}
\pagebreak
			\begin{remark}~
			\begin{enumerate}[(i)]
				\item Proposition \ref{lem:LANProperty} shows that our statistical model is asymptotically equivalent in Le Cam's sense to a standard Gaussian-shift model. It implies in particular that $\hat{\theta}_\nu$ is asymptotically efficient in the sense of Definition 11.1 of \cite{Ibragimov1981}.
				\item Without imposing Assumptions \ref{assump:generalA} and \ref{assump:specific_F} one can still show, combining Theorem 2.2, Proposition 2.1 and Theorem 2.1 of \cite{Tsybakov2009}, that the rate of convergence $\EV{\mathcal{I}_\nu}^{-1/2}$ is minimax-optimal in the sense of Definition 2.1 in \cite{Tsybakov2009}. We omit the details.
			\end{enumerate}
			\end{remark}
	
\subsection{\textbf{A numerical example}}\label{subsec:Numerics}

	\begin{figure}
	\includegraphics[width = \columnwidth]{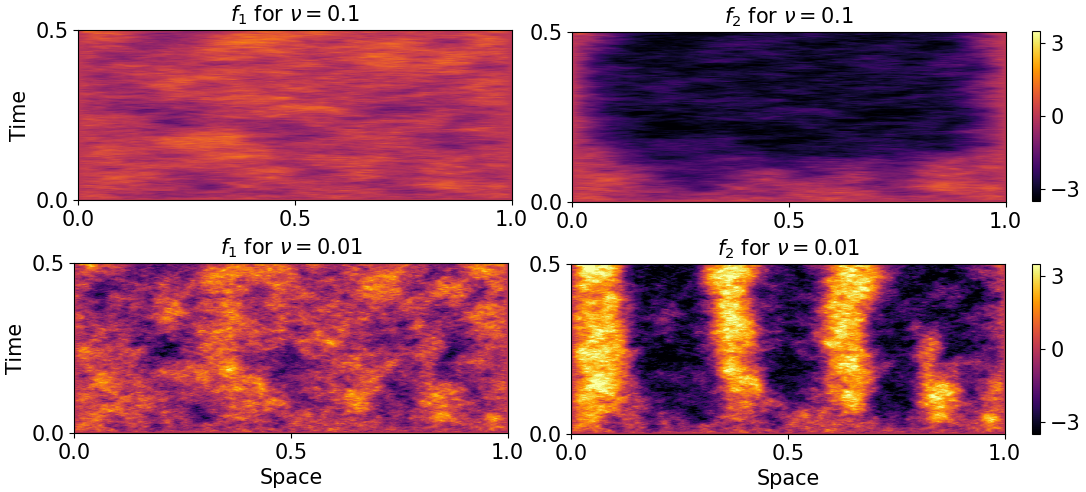}
	\caption{Realization of \eqref{eq:SPDE_sim} for $f_1(y) = -y(2+\sin(y))$ (left) and $f_2(y)=-(y^3-9y)$ (right) and $\nu\in \Set*{0.1,0.01}$.}
	\label{fig:Simulation}
	\end{figure}

	We simulate the solution to the stochastic heat equation \eqref{eq:Guiding_Example} with $\alpha=2$ and $\sigma\equiv 1$, i.e.
	\begin{equation}
		\D X_t(y) = \partial_{yy}^2 X_t(y)\,\D t + \theta f(X_t(y))\,\D t + \D W_t(y),\quad t\in[0,1],\quad X_0\equiv 0,\label{eq:SPDE_sim}
	\end{equation}
	on $\Lambda=(0,1)$. We choose $\theta=3$ and two different nonlinearities $f$, namely
	\begin{equation*}
		f_1(y) = -y(2+\sin(y))\text{ and }f_2(y) = -(y^3-9y),\quad y\in\mathbb{R}.
	\end{equation*}
	 Note that the latter case corresponds to the well-known $\Phi_1^4$ equation, an Allen-Cahn phase field model, with stable points $\pm3$. We use a semi-implicit Euler scheme with a finite difference approximation of $\partial_{yy}^2$ based on  \cite{Lord2014}. To this end, we discretise space-time using the uniform grid
	\begin{equation*}
		\Set*{(y_j,t_k)\given y_j= j/M, t_k=k/N, j=0,\dots,M, k=0,\dots,N}
	\end{equation*}
	for $M,N\in\mathbb{N}$. In line with the results of \cite{Lord2014} that we should choose $N\approx M^2$, we take $M=10^3$ and $N=10^6$. The results are displayed in Figure \ref{fig:Simulation}.
	
	All displays in Figure \ref{fig:Simulation} arise from the same realization of the noise $\D W_t$. Most strikingly, considering $f_2$, which gives rise to a double-well potential, supports our intuition on why the vanishing-diffusivity regime provides insight into the non-linearity: With the diffusivity at $\nu = 0.1$ only one stable point at $-3$ is visible and much information of $f_2$ is diluted. When the diffusivity is lower, however, the characteristic length scale shortens and the structure of two stable points at the states $\pm 3$ emerges.
	
	\begin{figure}[t]
	\centering
	\begin{subfigure}[t]{0.48\columnwidth}
	\centering
	\includegraphics[width=\columnwidth]{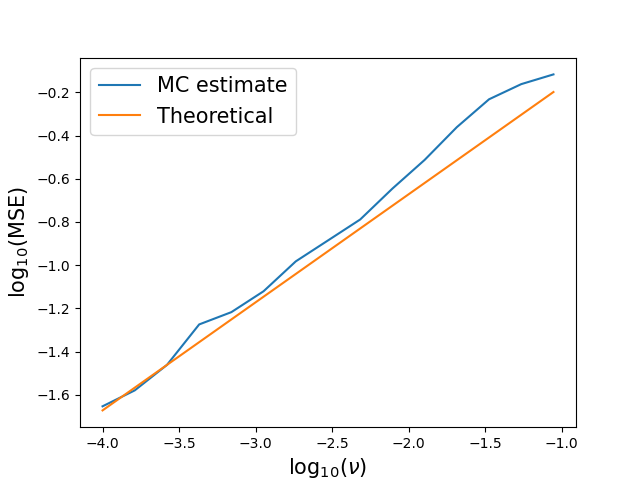}
	\end{subfigure}~
	\begin{subfigure}[t]{0.48\columnwidth}
	\centering
	\includegraphics[width=\columnwidth]{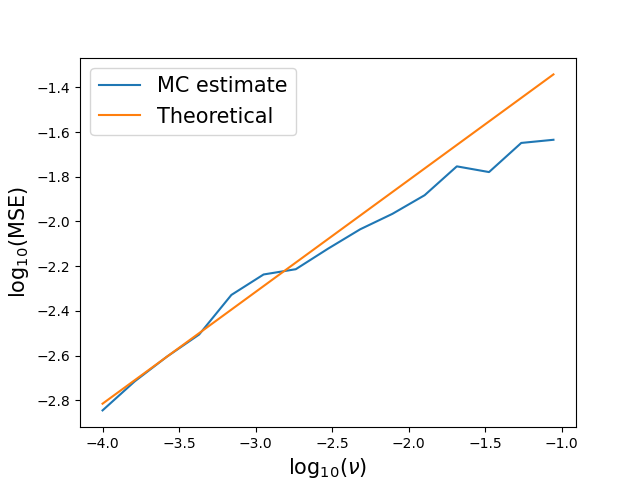}
	\end{subfigure}
	\caption{MSE as a function of the diffusivity $\nu$ for $f_1$ (left) and $f_2$ (right), $\log-\log$ plot.}\label{fig:MSE}
	\end{figure}
	
	 We implement the estimator $\hat{\theta}_\nu$ using the discretization given by \eqref{eq:Discretization_full} (below) with $\delta_y=1/M$ and $\delta_t=1/N$ and approximate $S_{\delta_{t,k}}X_{t_k}$ by a one-step implicit Euler scheme.  $300$ Monte-Carlo runs are computed for each value of $\nu$. In Figure \ref{fig:MSE}, we display the (Monte Carlo) MSE $\EV{(\hat{\theta}_\nu-\theta)^2}$.  In the case of $f_1$, the assumptions of Theorem \ref{thm:CLTnon-linear} are satisfied and the results reproduce well the (squared) theoretical convergence rate $\phi(\nu)^{-1}=\nu^{1/2}$. On the other hand, for $f_2$, the assumptions of Theorem \ref{thm:CLTnon-linear} are violated because $f_2$ does not allow for a linear growth bound. For diffusivities $\nu$ smaller than $0.01$ the  MSE follows the same theoretical rate of convergence as before. This might be explained by the locally linear growth of $f_2$ around its stable fixed points $\pm 3$.

\subsection{\textbf{Major steps in the proof of the Central Limit Theorem}}\label{subsec:towardsCLT}

	We rewrite the decomposition \eqref{eq:MLE_decmop} for the MLE $\hat{\theta}_\nu$ to obtain
	\begin{equation}\label{eq:decomthetanu}
		\EV*{\mathcal{I}_\nu}^{1/2}(\hat{\theta}_\nu-\theta) = \frac{\mathcal{M}_\nu}{\EV*{\mathcal{I}_\nu}^{1/2}}\frac{\EV*{\mathcal{I}_\nu}}{\mathcal{I}_\nu}.
	\end{equation}	
	Since $\mathcal{I}_\nu$ is the quadratic variation of the martingale $\mathcal{M}_\nu$, Theorem \ref{thm:CLTnon-linear} will follow from a martingale central limit theorem if we can show that
	\begin{equation}
		\frac{\EV*{\mathcal{I}_\nu}}{\mathcal{I}_\nu}\xrightarrow{\prob*{}}1\text{ as $\nu\to0$.}\label{eq:Claim_Main}
	\end{equation}
	
	\begin{claim}\label{claim:1}
		If $\EV*{\mathcal{I}_\nu}\to\infty$ as $\nu\to0$, then \eqref{eq:Claim_Main} holds.
	\end{claim}
	The second claim quantifies the rate of convergence $\EV*{\mathcal{I}_\nu}^{-1/2}$:
	\begin{claim}\label{claim:2}
		Under Assumption \ref{assump:general_F}, $\EV{\mathcal{I}_\nu}\sim\phi(\nu)\to \infty$ as $\nu\to 0$.
	\end{claim}
	The key mathematical challenge is to control the fluctuations of the observed Fisher information $\mathcal{I}_\nu$ around its expectation. Due to the presence of $F$, the dependence of $\mathcal{I}_\nu$ on the underlying Gaussian process $W_t$ is non-linear and very implicit.  By combining the Poincaré inequality from Gaussian analysis with the Da Prato - Debussche Trick to treat the non-linearity, we are able to prove Claims \ref{claim:1} and \ref{claim:2}. The proofs of the subsequent lemmas are postponed to Section \ref{sec:Proofs_CLT}.

	
		\subsubsection{Poincaré inequality}
		
		Given Assumption \ref{assump:specific_F}, the Poincaré inequality (Proposition 3.1 of \cite{Nourdin2009}) applies and yields
		\begin{equation*}
			\var\left(\mathcal{I}_\nu\right)\le\EV*{\norm*{\mathcal{D}\mathcal{I}_\nu}_{\mathfrak{H}}^2}.
		\end{equation*}
		Recalling that $\mathfrak{H}=L^2([0,T],\mathbb{H})$ and applying the chain rule of Malliavin calculus (Proposition 1.2.3 of \cite{Nualart2009}), it suffices to control
		\begin{equation*}
			\EV*{\norm*{\mathcal{D}\mathcal{I}_\nu}_{\mathfrak{H}}^2} = \EV*{\int_0^T \norm*{\int_\tau^T\int_\Lambda \frac{f(X_t(y))f'(X_t(y))}{\sigma(y)^2}\mathcal{D}_\tau X_t(y)\,\D y\,\D t}_{\mathbb{H}}^2\,\D \tau}.
		\end{equation*}
		The key estimate is given by the following result.
		\begin{lemma}\label{lem:boundedness_convergence}
			Grant Assumptions \ref{assump:generalA} and \ref{assump:specific_F}. For any $\phi\in L^2(\Lambda)$ with $\norm*{\phi}=1$, $\tau\in[0,T]$ and $t\in[0,T]$ we have
			\begin{equation*}
				\norm*{\int_\Lambda \phi(y)\mathcal{D}_\tau X_t(y)\,\D y}_{\mathbb{H}}\le \norm*{B}e^{\norm*{f'}_\infty t},
			\end{equation*}
independently of the realization of $X$.
		\end{lemma}
		\begin{lemma}\label{lem:Variance_Bound}
			Grant Assumptions \ref{assump:generalA}, \ref{assump:specific_F} and assume  $\EV*{\mathcal{I}_\nu}\to \infty$ as $\nu\to0$. Then $\var\left(\mathcal{I}_\nu\right)\lesssim \EV*{\mathcal{I}_\nu}$. In particular, Claim \ref{claim:1} holds.
		\end{lemma}	
		
\subsubsection{Da Prato - Debussche Trick}		

		The Da Prato-Debussche trick \cite{DaPrato2003} (also called splitting argument \cite{Altmeyer2020}) amounts to decomposing the solution $X$ of the semi-linear SPDE \eqref{eq:SPDE} into its linear part $\bar{X}$ and its non-linear part $\tilde{X}$:
		\begin{equation*}
			\bar{X}_t\coloneqq \int_0^t S_{\nu(t-s)} B \,\D W_s,\quad \tilde{X}_t\coloneqq X_t - \bar{X}_t.
		\end{equation*}
			Introduce
		\begin{equation*}
			\bar{\mathcal{I}}_{t,\nu}\coloneqq \int_0^t \norm*{B\inv\bar{X}_s}^2\,\D s.
		\end{equation*}
		Then Itô's isometry and \eqref{eq:compatibility} yield for any $t\in[0,T]$ the bounds
		\begin{equation}
			\EV*{\norm{\bar{X}_t}^2}=\int_0^t \norm*{S_{\nu(t-s)}B}_{\operatorname{HS}}^2\,\D s\begin{cases}\le \norm*{B}^2\phi(\nu,t)\\\ge\norm*{B\inv}^{-2} \phi(\nu,t)\end{cases}\sim \phi(\nu,t). \label{eq:linearpart_secondmoment}
		\end{equation}
		Consequently,
		\begin{equation}
			\EV*{\bar{\mathcal{I}}_{t,\nu}}= \int_0^t \int_0^s \norm*{B\inv S_{\nu(s-u)}B}_{\operatorname{HS}}^2\,\D u\D t \sim \int_0^t \phi(\nu,s)\,\D s \begin{cases}\lesssim \phi(\nu),\\\gtrsim t\phi(\nu).\end{cases}\label{eq:rate_linear}
		\end{equation}
	The proof of Claim \ref{claim:2} combines Assumption \ref{assump:general_F} with the Gronwall inequality, allowing us to control $\EV{\mathcal{I}_\nu}$ by its linear part $\EV{\bar{\mathcal{I}}_{T,\nu}}$.

	\begin{lemma}\label{lem:non-linearexpectation}
		Grant Assumptions \ref{assump:general}, \ref{assump:generalA} and \ref{assump:general_F}. Then we have $\EV*{\mathcal{I}_\nu}\sim \EV*{\bar{\mathcal{I}}_{T,\nu}}\sim \phi(\nu)$. In particular, Claim \ref{claim:2} holds.
	\end{lemma}
	
	\subsubsection{Proof of the CLT and its corollary}\label{subsec:ProofCLT}

	\begin{proof}[Proof of Theorem \ref{thm:CLTnon-linear}]
	Recall the decomposition \eqref{eq:Claim_Main} and that Lemma \ref{lem:Variance_Bound} shows $\EV{\mathcal{I}_\nu}/\mathcal{I}_\nu\xrightarrow{\prob*{}}1$ as $\nu\to0$. By a standard martingale central limit theorem (Theorem 5.5.4 of \cite{Liptser1989}) we conclude $\mathcal{M}_\nu/\EV{\mathcal{I}_\nu}^{1/2}\xrightarrow{d} N(0,1)$. Using Slutzky's Lemma, the asserted CLT follows.	
	The last claim follows from Lemma \ref{lem:non-linearexpectation}.
	\end{proof}

	\begin{proof}[Proof of Corollary \ref{cor:Confidence}]
		The convergence $\mathcal{I}_\nu/\EV{\mathcal{I}_\nu}\xrightarrow{\prob*{}} 1$ from Lemma \ref{lem:Variance_Bound}, Slutzky's Lemma and Theorem \ref{thm:CLTnon-linear} imply
		\begin{equation*}
			\mathcal{I}_\nu^{1/2}(\hat{\theta}_\nu-\theta)\xrightarrow{d}N(0,1),
		\end{equation*}
which yields the asserted asymptotic coverage probability of the confidence interval.
	\end{proof}
	
	\subsection{\textbf{Convergence rate under more general assumptions}}\label{subsec:GeneralRates}
	
	Assumptions \ref{assump:general}, \ref{assump:generalA} and \ref{assump:general_F} already suffice to derive the rate of convergence for $\hat{\theta}_\nu$.
	
\begin{proposition}\label{thm:non-linearrate}
		Grant Assumptions \ref{assump:general}, \ref{assump:generalA} and \ref{assump:general_F}. Then
		\begin{equation*}
			\hat{\theta}_\nu-\theta = \mathcal{O}_{\prob*{}}(\phi(\nu)^{-1/2}).
		\end{equation*}
	\end{proposition}


	The statement of Proposition\ref{thm:non-linearrate} also holds for general separable Hilbert spaces $\mathcal{H}$  without imposing that $A$ gives rise to a Green function, that $B^{-1}$ is bounded and that $F$ is of Nemytskii-type. The main assumption is just that the operators $A$ and $B$ share the same eigenspaces.

	\begin{proposition}\label{lem:general_rate}
		Let $\mathcal{H}$ be a separable Hilbert-space and let $B$ be a bounded linear operator. Denote $\mathcal{H}_B\coloneqq\operatorname{dom}(B\inv)$ with the inner product $\iprod*{}{}_B\coloneqq\iprod*{B\inv\MTemptyplaceholder}{B\inv \MTemptyplaceholder}$. Grant Assumption \ref{assump:general} and assume that $(X_t)_{t\in[0,T]}$ takes values in $\mathcal{H}_B$. Assume further that $A$ and $B$ have a common set of eigenfunctions $(\mu_k)_{k\in\mathbb{N}}$, forming a complete orthonormal system of $\mathcal{H}$, and  that $F\colon \mathcal{H}_B\to\mathcal{H}_B$ satisfies
		\begin{equation*}
			d\norm*{z}_{B}\le \norm*{F(z)}_B\le D(\norm*{z}_{B}+1),\quad z\in\mathcal{H}_B,
		\end{equation*}
		for constants $0<d\le D<\infty$ and $\EV{\norm*{X_0}_B^2}<\infty$. Then
		\begin{equation*}
			\hat{\theta}_\nu-\theta = \mathcal{O}_{\prob*{}}(\phi(\nu)^{-1/2}).
		\end{equation*}
	\end{proposition}

\section{Other observation schemes}\label{sec:discretize}
	
The MLE $\hat{\theta}_\nu$ from Lemma \ref{lem:MLE_decomp} uses full spatial and temporal information of the solution process $(X_t)_{t\in[0,T]}$. Here we show that it suffices to have local or discrete information in space and time. Local information also allow for non-parametric estimation of space-and time-varying $\theta(y,t)$. We assume throughout this section that $X_t\in C(\Lambda)$ for all $t\in (0,T]$. The proofs are postponed to Section \ref{sec:Obs_schemes_proofs}.
		
\subsection{\textbf{Localized perspective}}\label{subsec:Localize_Perspective}

Fix $y_0\in \Lambda$, $t_0\in (0,T)$, $0<\delta_y\le 1$ and $0<\delta_t\le \min(t_0,T-t_0)$. Set $\Lambda_{\delta_y}(y_0)\coloneqq y_0 + \delta_y(\Lambda-y_0)$ and define the localised  inner product
\[\iprod*{u}{v}_{\delta_y}\coloneqq \int_{\Lambda_{\delta_y}(y_0)} u(y)v(y)\,\D y,\quad u,v\in L^2(\Lambda_{\delta_y}(y_0)).\]

Fix a complete orthonormal system $(u_k)_{k\in\mathbb{N}}$ of $L^2(\Lambda_{\delta_y}(y_0))$ such that the extensions $(\tilde{u}_k)_{k\in\mathbb{N}}$ by zero belong to $\operatorname{dom}(A^\ast B\inv)$ for each $k\in\mathbb{N}$.
In the spirit of the MLE $\hat\theta_\nu$ define the localised estimator
		\begin{equation*}
			\hat{\theta}_{\nu,\delta_t,\delta_y}^{\operatorname{loc}}\coloneqq \frac{\int_{t_0-\delta_t}^{t_0+\delta_t}\iprod*{B\inv F(X_s)}{B\inv [\D X_s - \nu AX_s\,\D s]}_{\delta_y}}{\int_{t_0-\delta_t}^{t_0+\delta_t}\norm*{B\inv F(X_s)}_{\delta_y}^2\,\D s},
		\end{equation*}
where the numerator is defined by
		\begin{equation*}
\sum_{k=1}^\infty \left[ \int_{t_0-\delta_t}^{t_0+\delta_t} \iprod*{B\inv F(X_s)}{\tilde{u}_k}\D\iprod*{\tilde{u}_k}{B\inv X_s}-\nu\int_{t_0-\delta_t}^{t_0+\delta_t} \iprod*{B\inv F(X_s)}{\tilde{u}_k}\iprod*{X_s}{A^\ast B\inv \tilde{u}_k}\,\D s\right].
		\end{equation*}
Then $\hat{\theta}_{\nu,\delta_t,\delta_y}^{\operatorname{loc}}$ satisfies under ${\prob*{}}_\theta$ the representation
		\begin{equation*}
			\hat{\theta}_{\nu,\delta_t,\delta_y}^{\operatorname{loc}}-\theta = \frac{\int_{t_0-\delta_t}^{t_0+\delta_t}\iprod*{B\inv F(X_s)}{\D W_s}_{\delta_y}}{\int_{t_0-\delta_t}^{t_0+\delta_t}\norm*{B\inv F(X_s)}_{\delta_y}^2\,\D s}=\colon\frac{\mathcal{M}_{\nu,\delta_t,\delta_y}}{\mathcal{I}_{\nu,\delta_t,\delta_y}}.
		\end{equation*}
		Note that if $B\inv$, $F$ and $A^\ast$ are local operators, then $\hat{\theta}_{\nu,\delta_t,\delta_y}^{\operatorname{loc}}$ only requires observing $X$ on the space-time window $\Lambda_{\delta_y}(y_0)\times (t_0-\delta_t,t_0+\delta_t)$.
		
		We assume a scaling of the localised observed Fisher information in analogy to Claim \ref{claim:2} above. Two typical cases in which the following assumption holds will be presented in Lemmas \ref{lem:Assumptions_hold_a} and \ref{lem:Assumptions_hold_b} below.
		
\begin{assumption}\label{assump:Non_para_too_strong}
			We have
			\begin{equation*}
				c_1\delta_y^d\phi(\nu,t)\le \int_{\Lambda_{\delta_y}(y_0)}\EV*{X_t(y)^2}\,\D y \le c_2\delta_y^d\phi(\nu,t),\quad t\in (t_0-\delta_t,t_0+\delta_t),
			\end{equation*}
			with constants $0<c_1\le c_2<\infty$.
\end{assumption}

The following proposition shows that indeed local spatio-temporal information suffices to recover $\theta$.

\begin{proposition}\label{prop:CLT_Localized}
	Grant Assumptions \ref{assump:generalA} and \ref{assump:specific_F}. If $\EV*{\mathcal{I}_{\nu,\delta_t,\delta_y}}\to\infty$ as $\nu\to 0$, where $\delta_t$ and $\delta_y$ may vary with $\nu$, then
	\begin{equation*}
		\mathcal{I}_{\nu,\delta_t,\delta_y}^{1/2}(\hat{\theta}_{\nu,\delta_t,\delta_y}^{\operatorname{loc}}
-\theta)\xrightarrow{d}N(0,1)\text{ and } \mathcal{I}_{\nu,\delta_t,\delta_y}/\EV*{\mathcal{I}_{\nu,\delta_t,\delta_y}}\xrightarrow{\prob*{}}1.
	\end{equation*}
Under Assumptions \ref{assump:general_F} and \ref{assump:Non_para_too_strong} we have $\EV*{\mathcal{I}_{\nu,\delta_t,\delta_y}}\sim\delta_y^d \delta_t\phi(\nu)$.
\end{proposition}


Sufficient conditions for Assumption \ref{assump:Non_para_too_strong} to hold will be deduced from the fact that a local property of the Green function $G$ implies the required scaling for the linear part $\bar{X}_t$.

		\begin{lemma}\label{lem:Sclaing_Linear_Part}
		Assume that the kernel $G_t$ of $S_t$ satisfies for  $\delta_y^d\phi(\nu,t_0-\delta_t)$ large enough
		\begin{equation}
			c_3\delta_y^d\phi(\nu,t)\le \int_0^t\int_{\Lambda_{\delta_y}(y_0)}\norm*{G_{\nu(t-s)}(y,\MTemptyplaceholder)}^2\,\D y\D s \le c_4\delta_y^d\phi(\nu,t),\quad t\in (t_0-\delta_t,t_0+\delta_t), \label{eq:assump_nonpara_a}
		\end{equation}
		where the constants $0<c_3,c_4<\infty$ neither depend on $t,\nu$ nor $\delta_y$.
Then the linear part $\bar{X}_t$ satisfies
			\begin{equation*}
				c_3\norm*{B\inv}^{-2}\delta_y^d\phi(\nu,t)\le \int_{\Lambda_{\delta_y}(y_0)}\EV*{\bar{X}_t(y)^2}\,\D y \le c_4\norm*{B}^2\delta_y^d\phi(\nu,t),\quad t\in (t_0-\delta_t,t_0+\delta_t)
			\end{equation*}
			with $c_1(t)=c_3\norm*{B\inv}^{-2}$ and $c_2(t)=c_4\norm*{B}^2$.
		\end{lemma}
		The following two lemmas provide typical conditions under which the scaling of the linear part $\bar{X}_t$ dominates the non-linear part $\tilde{X}_t$, thereby implying Assumption \ref{assump:Non_para_too_strong}. The first lemma builds on uniform bounds of the non-linear part $\tilde{X}$ arising from the semi-group $(S_t)_{t\ge 0}$.

		\begin{lemma}\label{lem:Assumptions_hold_a}
		Grant Assumptions \ref{assump:generalA}, \ref{assump:specific_F}, the lower bound in \eqref{eq:assump_nonpara_a} and assume:
		\begin{enumerate}[(i)]
		\item the Green function $G_t(\MTemptyplaceholder,\MTemptyplaceholder)$ is non-negative on $\Lambda_{\delta_y}(y_0)\times\Lambda$ for all $t>0$ and satisfies
			\begin{equation*}
				\sup_{y\in\Lambda_{\delta_y}(y_0)}\norm*{G_t(y,\MTemptyplaceholder)}^2 \le c_5 \int_\Lambda\int_\Lambda G_t(y,\eta)^2\,\D y\D \eta= c_5\norm*{S_t}_{\operatorname{HS}}^2
			\end{equation*}
			for a constant $c_5$;
		\item there exists a constant $c_6>0$ such that for any $z\in C(\Lambda)$, $t\ge 0$, we have $\norm*{S_t z}_\infty\le c_6 \norm*{z}_\infty$;
		\item\label{item:assump_nonpara_t} $t_0\in [0,T)$ is small enough, such that $c_7(t_0+\Delta_t)>0$ for some $\Delta_t>0$ and
		\begin{equation*}
			c_7(t)\coloneqq \tfrac{c_3}{2\norm*{B\inv}^2} - 4\norm*{B}^2\abs*{\Lambda}e^{4Ttc_6c_8(\theta D)^2}c_5c_6c_8(\theta D)^2t^2,
		\end{equation*}
		where $c_8>0$ is an absolute constant defined in \eqref{eq:Riesz_Application} (below).
		\item $c_9\coloneqq \norm*{\EV{X_0^2}}_\infty<\infty$.
		\end{enumerate}
		Then Assumption \ref{assump:Non_para_too_strong} holds for $0<\delta_t\le \Delta_t$ and $\delta_y^d \phi(\nu,t_0-\delta_t)$ sufficiently large such that the second summand in \eqref{eq:largeenough} becomes negligible with $c_1=c_7(t_0+\Delta_t)$ and
		\[c_2=2\norm*{B}^2\abs*{\Lambda}c_5(1 + 4e^{4T (t_0+\delta_t)c_6c_8 (\theta D)^2}(t_0+\delta_t)^2 c_6c_8(\theta D)^2).\]
\end{lemma}

Assumption \ref{assump:Non_para_too_strong} also holds under a symmetry condition on the non-linearity $F$, which ensures $\EV{X_t}=0$ in the case $X_0=0$.

\begin{lemma}\label{lem:Assumptions_hold_b}
	Grant Assumptions \ref{assump:generalA}, \ref{assump:specific_F} and assume \eqref{eq:assump_nonpara_a} as well as
		\begin{enumerate}[(i)]
		\item  $F(-z)=-F(z)$ for any $z\in L^2(\Lambda)$;
		\item \label{item:assump_nonpara_t_2}$t_0\in [0,T)$ is small enough, such that $\tilde{c}_7(t_0+\Delta_t)>0$ for some $\Delta_t>0$ and
		\begin{equation*}
			\tilde{c}_7(t)\coloneqq \tfrac{c_3}{2\norm*{B\inv}^2} - c_4\norm*{B}^2e^{2\norm*{f'}_\infty t}\norm*{f'}_\infty^2 t^2;
		\end{equation*}
		\item $X_0=0$.
		\end{enumerate}
		Then Assumption \ref{assump:Non_para_too_strong} holds for $0<\delta_t\le\Delta_t$ with $c_1=\tilde{c}_7(t_0+\Delta_t)$ and $c_2=2c_4\norm*{B}^2(1 + e^{2\norm*{f'}_\infty (t_0+\delta_t)}\norm*{f'}_\infty^2 (t_0+\delta_t)^2)$.
		\end{lemma}

\begin{remark}
			Conditions \eqref{item:assump_nonpara_t} and \eqref{item:assump_nonpara_t_2} in Lemmas \ref{lem:Assumptions_hold_a} and \ref{lem:Assumptions_hold_b} ensure that at all times $t\in (t_0-\delta_t,t_0+\delta_t)$ the exponential growth bounds for $\tilde{X}$, arising from applying Gronwall's inequality, are smaller than the growth of the linear part $\bar{X}$.
\end{remark}

\begin{lemma}\label{lem:Assumptions_hold_Example}
			 The fractional heat equation \eqref{eq:Guiding_Example} satisfies Condition \eqref{eq:assump_nonpara_a} with $c_3$ from \eqref{eq:c_3} below and $c_4=2$. It satisfies Assumption \ref{assump:Non_para_too_strong} for $t_0$ such that \eqref{item:assump_nonpara_t} of Lemma \ref{lem:Assumptions_hold_a} holds with $c_5=2/l$, $c_6=1$, $c_9=0$ and $c_8=1$. If $f(-y)=-f(y)$, $y\in \mathbb{R}$, then it satisfies Assumption \ref{assump:Non_para_too_strong} for all $t_0$ such that \eqref{item:assump_nonpara_t_2} of Lemma \ref{lem:Assumptions_hold_b} holds.
\end{lemma}

\subsection{\textbf{Non-parametric estimation}}\label{subsec:Nonparametric}
	
We consider the situation that $\theta=\theta(y,t)$ depends on both, space and time, and we aim at estimating $\theta(y_0,t_0)$ non-parametrically for a given time-space point $(y_0,t_0) \in \Lambda\times(0,T)$. We assume  $\sup_{(y,t)\in  \Lambda\times[0,T]}\abs*{\theta(y,t)}\le \theta_{\max}<\infty$ and that $t\mapsto\theta(y,t)$ is $\eta_t$-Hölder-continuous uniformly in a neighbourhood of $y_0$ and $y\mapsto\theta(y,t)$ is $\eta_y$-Hölder-continuous uniformly in a neighbourhood of $t_0$. Recall the notation $[\theta_t F(z)](y)=\theta(y,t)f(z(y))$ for $z\in L^2(\Lambda)$, $y\in\Lambda$. The proofs are postponed to Section \ref{subsec:Proof_Nonparametric}.
		
		In this more general setting the estimation error permits under ${\prob*{}}_\theta$ the representation
		\begin{equation*}
			\hat{\theta}_{\nu,\delta_t,\delta_y}^{\operatorname{loc}}-\theta(y_0,t_0) = \frac{\int_{t_0-\delta_t}^{t_0+\delta_t}\iprod*{B\inv \theta_t F(X_t)}{B\inv F(X_t)}_{\delta_y}\D t}{\int_{t_0-\delta_t}^{t_0+\delta_t}\norm*{B\inv F(X_t)}_{\delta_y}^2\,\D t}-\theta(y_0,t_0)+ \frac{\mathcal{M}_{\nu,\delta_t,\delta_y}}{\mathcal{I}_{\nu,\delta_t,\delta_y}}.
		\end{equation*}
	\begin{theorem}\label{thm:nonparaCLT}
		Grant Assumptions \ref{assump:generalA} and \ref{assump:specific_F}. Then
		\begin{equation*}
			\hat{\theta}_{\nu,\delta_t,\delta_y}^{\operatorname{loc}}-\theta(y_0,t_0) = \mathcal{O}\left(\delta_t^{\eta_t} + \delta_y^{\eta_y}\right) + \mathcal{I}_{\nu,\delta_t,\delta_y}^{-1/2}Z_\nu 
		\end{equation*}
holds with $Z_\nu \xrightarrow{d} N(0,1)$ as $\nu\to0$, provided $\EV*{\mathcal{I}_{\nu,\delta_t,\delta_y}}\to\infty$. In this case we have $\mathcal{I}_{\nu,\delta_t,\delta_y}/\EV{\mathcal{I}_{\nu,\delta_t,\delta_y}}\xrightarrow{\prob*{}}1$.
		
		Under Assumptions \ref{assump:general_F} and \ref{assump:Non_para_too_strong} we have $\mathcal{I}_{\nu,\delta_t,\delta_y}\sim\delta_y^d\delta_t\phi(\nu)$.
\end{theorem}

If $\delta_t,\delta_y\to 0$ so fast that the approximation error becomes negligible, we arrive at the same CLT as before also in the non-parametric setting.

\begin{corollary}\label{cor:nonparaCLT}
		Grant Assumptions \ref{assump:generalA}, \ref{assump:specific_F}, \ref{assump:general_F} and \ref{assump:Non_para_too_strong}. If  $\delta_y^d\delta_t\phi(\nu)\to\infty$, $\delta_y^{d+2\eta_y}\delta_t\phi(\nu,t)\to0$ and $\delta_y^d\delta_t^{1+2\eta_t}\phi(\nu)\to0$, then
\begin{equation*}			\mathcal{I}_{\nu,\delta_t,\delta_y}^{1/2}(\hat{\theta}_{\nu,\delta_t,\delta_y}^{\operatorname{loc}}-\theta(y_0,t_0))\xrightarrow{d} N(0,1)\text{ as }\nu\to 0.
\end{equation*}
\end{corollary}

When we balance approximation and stochastic error by the right choice of $\delta_t,\delta_y$, then we obtain convergence rates, which are well known for non-parametric estimation over anisotropic H\"older classes, see e.g. \cite{Hoffman2002}. These are formulated in terms of  the so-called effective smoothness $\eta_{\operatorname{eff}}$ of $(y,t)\mapsto \theta(y,t)$, defined via
		\begin{equation*}
			\frac{1}{\eta_{\operatorname{eff}}}\coloneqq \frac{d}{\eta_y}+\frac{1}{\eta_t}.
		\end{equation*}

\begin{corollary}\label{cor:nonparaOptimalChoice}
		Grant Assumptions \ref{assump:generalA}, \ref{assump:specific_F}, \ref{assump:general_F} and \ref{assump:Non_para_too_strong}. For the (asymptotically optimal) choices $\delta_y=\phi(\nu)^{-\eta_{\operatorname{eff}}/(\eta_y(2\eta_{\operatorname{eff}} +1))}$ and  $\delta_t=\phi(\nu)^{-\eta_{\operatorname{eff}}/(\eta_t(2\eta_{\operatorname{eff}} +1))}$
we obtain
\begin{equation*}
\hat{\theta}_{\nu,\delta_t,\delta_y}^{\operatorname{loc}}-\theta(y_0,t_0)=\mathcal{O}_{\prob*{}}
\left(\phi(\nu)^{-\eta_{\operatorname{eff}}/(2\eta_{\operatorname{eff}} +1)}\right)\text{ as }\nu\to 0.
\end{equation*}
\end{corollary}

\begin{remark}
		Lemmas \ref{lem:Assumptions_hold_a} and \ref{lem:Assumptions_hold_b}, ensuring that Assumption \ref{assump:Non_para_too_strong} holds, carry over to the non-parametric case if we replace $\theta$ in \eqref{item:assump_nonpara_t} of Lemma \ref{lem:Assumptions_hold_a} by the upper bound $\theta_{\operatorname{max}}$.
\end{remark}
%

\subsection{\textbf{Discrete Observations}}
	
	As a proof of concept, we propose an estimator that only uses discrete observations of $X$, but retains the asymptotic properties of the MLE $\hat{\theta}_\nu$. We do not strive for minimal assumptions on the mesh size of the discretisation. We will crucially rely on controlling the spatial regularity of $(X_t)_{t\in[0,T]}$ as $\nu\to 0$. Recall that even $\EV{\norm{X_t}^2}\to\infty$ as $\nu\to 0$ from (the proof of) Lemma \ref{lem:non-linearexpectation}. For simplicity, we assume that $B$ is self-adjoint.

Let $C^\eta\coloneqq \Set{f\colon \Lambda\to\mathbb{R}\given \norm*{f}_{C^\eta}<\infty}$, $\eta\in (0,1]$, denote the space of $\eta$-Hölder continuous functions for
			\begin{equation*}
				\norm*{f}_{C^\eta}\coloneqq \norm*{f}_\infty + \sup_{x\neq y\in \Lambda}\frac{\abs*{f(x)-f(y)}}{\abs*{x-y}^\eta}.
			\end{equation*}
For $\eta\in (0,1)$ and $p\ge 1$ let $W^{\eta,p}\coloneqq\Set{g\in L^p(\Lambda)\given \norm*{g}_{W^{\eta,p}} <\infty}$ be the $L^p$-Sobolev space of regularity $\eta$ (see Section 2 of \cite{DiNezza2012}) with norm
			\begin{equation*}
				\norm*{g}_{W^{\eta,p}}\coloneqq \left(\norm*{g}_{L^p(\Lambda)}^p + \int_\Lambda\int_\Lambda \frac{\abs*{g(x)-g(y)}^p}{\abs*{x-y}^{d+\eta p}}\,\D x\D y\right)^{1/p}.
			\end{equation*}

\begin{assumption}[$\gamma,p$]\label{assump:discretize}
			$(-A)$ is self-adjoint with an eigensystem $(\mu_k,w_k)_{k\in\mathbb{N}}$, where $ck^\beta\le \mu_k\le C k^\beta$ for some $\beta>1$, $0<c\le C<\infty$ with Green function $G_{\MTemptyplaceholder}(y,\MTemptyplaceholder)\in \mathfrak{H}$ such that $S_t z(y)=\iprod*{G_t(y,\MTemptyplaceholder)}{z}$ for all $y\in\Lambda$, $z\in L^2(\Lambda)$, $t\in(0,T]$. Let $0\le \gamma<1$, $p\ge 2$, and assume the following:
			\begin{enumerate}[(i)]
				\item $(-A)$ generates an analytic semi-group on $L^p(\Lambda)$.  $\mathcal{B}_{\gamma,p}$ denotes the corresponding interpolation space with norm $\norm{}_{\gamma,p}$ (following Section 4.4 of \cite{Hairer2009} and note  $\mathcal{B}_{0,p}=L^p(\Lambda)$);
				\item $\EV{\norm{X_0}_{0,p}^2}<\infty$;
				\item $F\colon L^p(\Lambda)\to L^p(\Lambda)$ is $M$-Lipschitz and
				\begin{equation*}
					\norm*{F(z)}_{\gamma,p} \le C (\norm*{z}_{\gamma,p}+1),\quad z\in \mathcal{B}_{\gamma,p},
				\end{equation*}
				for a constant $C>0$.
				\item There exists $0<\tilde{\gamma}< 1$ depending on $\gamma$, such that $\mathcal{B}_{\gamma,p}\hookrightarrow W^{\tilde{\gamma},p}$ continuously;
				\item The Green function $G_t(\MTemptyplaceholder,\MTemptyplaceholder)$ satisfies uniformly in $t>0$
				\begin{equation*}
					\sup_{x\in\Lambda}\norm*{G_t(x,\MTemptyplaceholder)}^2 \lesssim \int_\Lambda\int_\Lambda G_t(x,y)^2\,\D x\D y= \norm*{S_t}_{\operatorname{HS}}^2.
					\end{equation*}
				\end{enumerate}
		\end{assumption}

		Clearly, Assumption \ref{assump:discretize} implies Assumption \ref{assump:general}.

		\begin{remark}~
		\begin{enumerate}[(i)]
			\item If $\mathcal{B}_{\gamma,p}= W^{\tilde{\gamma},p}$, then a sufficient condition for the bound $\norm*{F(z)}_{\gamma,p}\lesssim \norm*{z}_{\gamma,p}$ is given by $F(z)=f\circ z$ and $f$ being $M$-Lipschitz.

			\item In particular, under Assumption \ref{assump:discretize}, we have $\phi(\nu,t) \sim \nu^{-1/\beta}t^{1-1/\beta}$. Combining this with (the proof of) Lemma \ref{lem:non-linearexpectation} gives the bound
			\begin{equation}
				\EV*{\norm*{B\inv F(X_t)}^2}\lesssim \nu^{-1/\beta}t^{1-1/\beta} +1.\label{eq:growthbound_L^2}
			\end{equation}
			\end{enumerate}
		\end{remark}

\begin{example}[continued]
Assumption \ref{assump:discretize} is satisfied in the case of the fractional heat equation \eqref{eq:Guiding_Example} with $\beta=\alpha$ and for arbitrarily large $p\ge 2$.
			Part (i) follows from Theorems 2.12 and 2.23 (note also Remark 2.7) of \cite{Yagi2010}. Theorem 16.15 of \cite{Yagi2010} shows that  $\mathcal{B}_{\gamma,p}$ is isomorphic to $W^{\tilde{\gamma},p}$ for $\tilde{\gamma}=\alpha\gamma$ and yields parts (i) and (iv). (ii) is trivial for $X_0\equiv 0$. (iii) is satisfied by our smoothness assumptions on $f$ and $\sigma$.  Finally, (v) is satisfied due to \eqref{eq:frac_Laplace_kernelnorm} (below).
		\end{example}
		
		In Corollary \ref{cor:Höldernorm} (below) we combine regularity results for semi-groups on interpolation spaces with Sobolev and Hölder embeddings to control the spatial regularity of $(X_t)_{t\in(0,T]}$. Whereas the bounds for fixed $\nu>0$ are standard (e.g.\ Theorem 5.25 of \cite{DaPrato2014} and Theorem 5.10 of \cite{Hairer2009} or Proposition 32 of \cite{Altmeyer2020}), we provide quantification of the spatial Hölder norms as $\nu\to 0$, obtaining:
		
\begin{proposition}\label{lem:RegularityResult}
			Grant Assumption \ref{assump:discretize} with either $0\le \gamma<(1-1/\beta)/2-1/p$ and $\tilde{\gamma}-d/p>0$ or $0\le \gamma<(1-1/\beta)/2$ and $p=2$, $\tilde{\gamma}>d/2$. Then
			\begin{equation*}
				\EV*{\norm*{X_t}_{C^{\tilde{\gamma}-d/p}}^2}\lesssim \EV*{\norm*{X_t}_{\gamma,p}^2}\lesssim \nu^{-2\gamma}(\phi(\nu) + t^{-2\gamma}),\quad t\in(0,T].
			\end{equation*}
\end{proposition}

For $p=2$, the Hilbert space structure allows for sharper estimates, which are required for controlling the error bounds for the temporal discretization. The phenomenon that the factorization method may provide sub-optimal spatial regularity for small $p$ is known, e.g.\  Theorem 5.25 of \cite{DaPrato2014} does not provide additional regularity for the stochastic heat equation and $p=2$.

We return to the task of approximating $\hat{\theta}_\nu$ based on discrete observations. For simplicity, we restrict ourselves to the case  $\Lambda=(0,1)$ with discrete observations
		\begin{equation*}
			(X_{t_k}(y_l))_{\substack{k=1,\dots,n\\ l=0,\dots,m}}\quad \text{with}\quad 0=t_0<\dots<t_n=T,\quad 0=y_0<\dots <y_m=1.
		\end{equation*}
		Set $\delta_{t,k}\coloneqq t_{k+1}-t_k$, $k=0,\dots, n-1$, $\delta_t\coloneqq \max_k\delta_{t,k}$ and $\delta_y\coloneqq \max_l(y_{l+1}-y_l)$. The time grid is supposed to be homogeneous in the sense that for some $C\ge 1$, $\delta_t\le C\delta_{t,k}$ uniformly over $k$ and $\nu$. Furthermore, in order to have sufficient spatial regularity of $X_{t_1}$, assume that $\delta_{t,0}\ge C'\nu^{1/(2\beta\gamma)}$ for some $C'>0$ independent of $\nu>0$.
		
		As an intermediate step to a fully discretised estimator, consider the case where we have continuous observations in space, i.e.\ $X_{t_1},\dots, X_{t_n}$. Then define the time discretized estimator
		\begin{equation}
			\hat{\theta}_{\nu,\delta_t }^{\operatorname{disc}} \coloneqq \frac{ \sum_{k=1}^{n-1} \delta_{t,k}\iprod*{B\inv F(X_{t_k})}{B\inv \frac{X_{t_{k+1}}-S_{\nu\delta_{t,k}} X_{t_k}}{\delta_{t,k}}}}{  \sum_{k=1}^{n-1} \delta_{t,k}\norm*{B\inv F(X_{t_k})}^2},\label{eq:Discretization_temp}
		\end{equation}
		similar to \citep{Hildebrandt2021b}. Note that we use the semi-group $S_t$ to integrate forward in time, whence we will not rely on temporal regularity of $X$.
		
		In order to discretise in space, we denote by $L^{\delta_y}\colon C(\Lambda)\to C(\Lambda)$ the piece-wise linear approximation operator such that
			\begin{equation*}
				[L^{\delta_y}f](y_l) = f(y_l),\quad l=0,\dots,m,
			\end{equation*}
		and linear interpolation in between.
		Then  the temporally and spatially discretised estimator is given by
		\begin{equation}
			\hat{\theta}_{\nu,\delta_t ,\delta_y}^{\operatorname{disc}}\coloneqq \frac{\sum_{k=1}^{n-1} \delta_{t,k}\iprod*{B\inv F(L^{\delta_y} X_{t_k})}{B\inv \frac{L^{\delta_y} X_{t_{k+1}}-S_{\nu\delta_{t,k}} L^{\delta_y} X_{t_k}}{\delta_{t,k}}}}{\sum_{k=1}^{n-1} \delta_{t,k}\norm*{B\inv F(L^{\delta_y} X_{t_k})}^2}.\label{eq:Discretization_full}
		\end{equation}
		 The main difficulty for the given spatial discretization is that the commutator between the interpolation operator $L^\delta$ and the semi-group $S_t$ is not easy to control
		 
In order to bound the deviation of $\hat{\theta}_{\nu,\delta_t ,\delta_y}^{\operatorname{disc}}$ from $\hat{\theta}_\nu$, we use Assumption \ref{assump:discretize} for pairs $(p_t,\gamma_t)$ (for the temporal discretisation) and $(p_y,\gamma_y)$ (for the spatial discretisation). Larger values of $\gamma_t$, $\gamma_y$ and $p_y$ will allow for larger mesh sizes. We combine Propositions \ref{prop:discret_temp} and \ref{prop:discret_spat} (below) to obtain:

\begin{theorem}\label{thm:Discret}
			Grant Assumption \ref{assump:discretize} with $0<\gamma_t<(1-1/\beta)/2$ and $p_t=2<1/\tilde{\gamma}_t$ as well as $0<\gamma_y<(1-1/\beta)/2-1/p_y$ and $p_y>1/\tilde{\gamma}_y$. Suppose that $\delta_t =\smallo(\nu^{1/(2\beta\gamma_t)})$, $\delta_y= \smallo(\nu^{(\gamma_y+1/(2\beta))/(\tilde{\gamma}_y-1/p_y)}\delta_t^{1/(\tilde{\gamma}_y-1/p_y)})$ and $\norm{B^{-2} F(z)}_{\gamma,2}\lesssim \norm{z}_{\gamma,2}+1$, $z\in \mathcal{B}_{\gamma,2}$. Then
			\begin{equation*}
				\hat{\theta}_\nu - \hat{\theta}_{\nu,\delta_t ,\delta_y}^{\operatorname{disc}}=\smallo_{\prob*{}}(\phi(\nu)^{-1/2}).
			\end{equation*}
In particular, Theorem \ref{thm:CLTnon-linear} and Proposition \ref{thm:non-linearrate} with the CLT and the estimation rate continue to hold for $\hat{\theta}_{\nu,\delta_t ,\delta_y}^{\operatorname{disc}}$ instead of $\hat{\theta}_\nu$.
\end{theorem}

\begin{example}[continued]
		Consider the stochastic heat equation with non-linear reaction from \eqref{eq:Guiding_Example} with $\alpha=\beta=2$. We have already seen that Assumption \ref{assump:discretize} holds for any $p\ge 2$. We can choose $\gamma_t$ and $\gamma_y$ arbitrarily close (but strictly less than) $1/4$ and $\tilde{\gamma}_k=2\gamma_k$, $k\in\Set{t,y}$. Then the conclusions of Theorem \ref{thm:Discret} hold whenever
			\begin{equation*}
				\delta_t \sim \nu^{1+\epsilon_1},\quad \delta_y \sim \nu^{1+\epsilon_2}\delta_t^{2+\epsilon_3}\sim \nu^{3+\epsilon_1+\epsilon_2+\epsilon_3}
			\end{equation*}
for any $\epsilon_1,\epsilon_2,\epsilon_3>0$.
\end{example}

%
%

\begin{appendix}	

\section{Proofs for Section \ref{sec:Intro}}\label{subsec:Proof_Lemma_Example}
	
\begin{proof}[Proof of Lemma \ref{lem:examples_Assumptions}]
		
	First we check that Assumption \ref{assump:generalA} is satisfied. Using the explicit form of the eigenvalues $\mu_k=\lambda_k^{\alpha/2}=(k\pi/l)^\alpha$, $k\in\mathbb{N}$, of $(-\Updelta)^{\alpha/2}$, we find
		\begin{equation*}
			\phi(\nu,t)=\int_0^t \norm*{S_{\nu s}}_{\operatorname{HS}}^2\,\D s= \sum_{k\in\mathbb{N}} \int_0^te^{-2\mu_k\nu s}\D s \sim \sum_{k\in\mathbb{N}} \big((\nu k^\alpha)^{-1}\wedge t\big)\sim \nu^{-1/\alpha}t^{1-1/\alpha}.
		\end{equation*}
		Note that the eigenfunctions $(e_k)_{k\in\mathbb{N}}$ are smooth and uniformly bounded in supremum norm. Hence, the fractional heat kernel is given by
		\begin{equation}
			G_t(x,y) = \sum_{k\in\mathbb{N}} e^{-\lambda_k^{\alpha/2} t}e_k(x)e_k(y),\quad t\ge 0, x,y\in[0,l],\label{eq:fractional_kernel}
		\end{equation}
		with convergence in supremum norm.
		
		We now proceed to Assumption \ref{assump:specific_F}. As $(-\Updelta)^{\alpha/2}$ is the generator of a subordinated stopped Brownian motion \cite{Lischke2020}, $G_{\nu t}(x,y)\D y$ defines a measure on $\Lambda$ for all $\nu ,t>0$, $x\in\Lambda$, whose total mass is uniformly (over $(\nu t,x)$) bounded by 1. Furthermore, we have
		\begin{equation}
			\sup_{x\in\Lambda}\norm*{G_{\nu t}(x,\MTemptyplaceholder)}^2 \le \frac{2}{l}\sum_{k\in\mathbb{N}} e^{-2\lambda_k^{\alpha/2} \nu t} = \frac{2}{l}\norm*{S_{\nu t}}_{\operatorname{HS}}^2,\quad \nu, t\ge 0.\label{eq:frac_Laplace_kernelnorm}
		\end{equation}
		Two Picard iterations as in Theorems 6.2 and 7.1 of \cite{SanzSole2005} show that these properties ensure that Assumption \ref{assump:specific_F} is satisfied.
		\end{proof}
		
	\section{Proofs for Section \ref{sec:MainResults}}\label{sec:Proofs_Main}

	\subsection{\textbf{Proof of Lemma \ref{lem:MLE_decomp}}}\label{subsec:proof_MLE}
	
			Assumption \ref{assump:general} ensures that the Girsanov Theorem for semi-linear SPDEs, Theorem 10.18 of \cite{DaPrato2014}, is applicable and we will follow Remark 10.19 of \cite{DaPrato2014}. To this end, let $(Z_t)_{t\in[0,T]}$ be the unique weak solution to the heat equation $\D Z_t = \nu AZ_t + B\D \tilde{W}_t$, $Z_0=X_0$, driven by a $\prob*{}$-cylindrical Brownian motion $\tilde{W}$ and define the probability measure $\qrob*[\theta]{}$ on $(\Omega,\mathcal{F})$  via its Radon-Nikodym derivative
			\begin{equation*}
				\qrob*[\theta]{\D\omega}\coloneqq \exp{\int_0^T \iprod*{B\inv\theta F(X_t)}{\D \tilde{W}_t}(\omega) -\frac{1}{2}\int_0^T \norm*{B\inv\theta F(X_t(\omega))}^2\,\D t}\prob*{\D\omega}.
			\end{equation*}
			Theorem 10.18 of \cite{DaPrato2014} shows that $W_t\coloneqq \tilde{W}_t - \theta\int_0^t B\inv F(X_s)\,\D s$ is a $\qrob*[\theta]{}$-cylindrical Brownian motion. Moreover, $Z_t$ is a solution to the semi-linear SPDE \eqref{eq:SPDE} under $\qrob*[\theta]{}$. By uniqueness (in law) of the solution to \eqref{eq:SPDE}, the push-forward measure of $\qrob*[\theta]{}$ under $Z$ coincides with $\prob*[\theta]{}$ and it therefore suffices to find an expression for
			\begin{equation*}
				L(\theta,X)\coloneqq \EV*{\exp{\theta\int_0^T\iprod*{B\inv F(X_t)}{\D \tilde{W}_t} - \frac{\theta^2}{2}\int_0^T \norm*{B\inv F(X_t)}^2\,\D t}\given (X_t)_{t\in[0,T]}},
			\end{equation*}
			depending only on the path $(X_t)_{t\in[0,T]}$, and to maximise it with respect to $\theta$. To this end, take a complete orthonormal system $(u_k)_{k\in\mathbb{N}}$ such that $B\inv u_k\in\operatorname{dom}(A)$ for all $k\in\mathbb{N}$. Coefficient-wise computations show with convergence in $L^2(\prob*{})$
			\begin{equation}
				\begin{split}\int_0^T \iprod*{B\inv F(X_t)}{\D \tilde{W}_t} = \begin{multlined}[t]
	\sum_{k=1}^\infty\left[\int_0^T\iprod*{B^{-1}F(X_t)}{u_k}\D\iprod*{u_k}{B^{-1}X_t}\right.\\
	\left.-\nu\int_0^T \iprod*{B^{-1}F(X_t)}{u_k}\iprod*{X_t}{A^\ast B\inv u_k}\,\D t\right].
	\end{multlined}
	\end{split}\label{eq:likelihood}
			\end{equation}
This expression only depends on the path $X$ and we maximise the likelihood $L(\theta,X)$ with respect to $\theta\in{\mathbb R}$ to obtain
			\begin{equation*}
				\hat{\theta}_\nu = \frac{\int_0^T \iprod{B\inv F(X_t)}{\D\tilde{W}_t}}{\int_0^T \norm*{B\inv F(X_t)}^2\,\D t}.
			\end{equation*}
			Substituting $\tilde{W}_t = W_t + \theta\int_0^T B\inv F(X_t)\,\D t$, where $W_t$ is a cylindrical Brownian motion under $\qrob*[\theta]{}$, yields the last claim.\hfill\qed

		
	\subsection{\textbf{Proof of Proposition \ref{lem:LANProperty}}}\label{subsec:Proof_LAN}
	
Under $\prob*[\theta]{}$ we have
\begin{align*}
&\ell(X,\theta+h_\nu)-\ell(X,\theta)\\
&\overset{d}{=}	\int_0^T(\theta + h_\nu)\iprod*{B\inv F(X_t)}{\D W_t + \theta B\inv F(X_t)\,\D t}-		\frac{1}{2}\int_0^T (\theta+h_\nu)^2 \norm*{B\inv F(X_t)}^2\,\D t\\
&\quad -\int_0^T\theta\iprod*{B\inv F(X_t)}{\D W_t + \theta B\inv F(X_t)\,\D t}+
						\frac{1}{2}\int_0^T \theta^2 \norm*{B\inv F(X_t)}^2\,\D t\\
						&= h_\nu\int_0^T \iprod*{B\inv F(X_t)}{\D W_t}-\frac{1}{2}h_\nu^2 \int_0^T \norm*{B\inv F(X_t)}^2\,\D t.
\end{align*}
Exactly as in the proof of Theorem \ref{thm:CLTnon-linear}, applying a martingale CLT (Theorem 5.5.4 of \cite{Liptser1989}) to the first term, Slutzky's Lemma and Lemma \ref{lem:Variance_Bound} yields the claimed convergence.\hfill \qed
		
\subsection{\textbf{Proofs for Subsection \ref{subsec:towardsCLT}}}\label{sec:Proofs_CLT}
	
\subsubsection{Poincaré inequality}
	In order to prove Lemma \ref{lem:boundedness_convergence}, we will use a Picard iteration to approximate $\mathcal{D}X_t(y)$. For fixed $\tau\in[0,T]$, $t\in(\tau,T]$, define the sequence
	\begin{equation}
		\begin{cases}u_{t,\tau}^{(0)}(y)\coloneqq G_{\nu(t-\tau)}(y,\MTemptyplaceholder),\\
		u_{t,\tau}^{(n)}(y)\coloneqq u_{t,\tau}^{(0)}(y) + \int_\tau^t \int_\Lambda G_{\nu(t-s)}(y,\eta)f'(X_s(\eta))u_{s,\tau}^{(n-1)}(\eta)\,\D\eta\D s,
		\end{cases}\label{eq:sequence_diff}
	\end{equation}
	$y\in\Lambda$, $n\in\mathbb{N}$, taking values in $\mathbb{H}$. For $t\in[0,\tau]$, $u_{t,\tau}^{(n)}(y) \equiv 0$ for all $y\in\Lambda$, $n\in\mathbb{N}_0$.
			\begin{lemma}\label{lem:Picard2}
			Grant Assumptions \ref{assump:generalA} and \ref{assump:specific_F}. For all $\phi\in L^2(\Lambda)$ with $\norm*{\phi}=1$, $\nu>0$, $\tau\in[0,T]$, $t\in(\tau,T]$ and  all realizations $(X_s)_{s\in(\tau,t]}$, the sequence $(u_{t,\tau}^{(n)})_{n\in\mathbb{N}}$ given by \eqref{eq:sequence_diff} satisfies
				\begin{equation*}
					K_n\coloneqq \norm*{\int_\Lambda \phi(y)u_{t,\tau}^{(n)}(y)\,\D y}_{\mathbb{H}} \le \norm*{B} + \sum_{m=1}^n \frac{\norm*{B}}{m!}\norm*{f'}_\infty^{m}(t-\tau)^{m}.
				\end{equation*}
				In particular, we have $K_n\le \norm*{B}e^{\norm*{f'}_\infty t}<\infty$ for all $\tau$, $t$, $n$, $\phi$ and realizations of $(X_t)_{t\in(\tau,T]}$.
\end{lemma}

\begin{proof}[Proof (by induction)] We  have
			\begin{align*}
				K_0^2 &=\norm*{\int_\Lambda\phi(y)G_{\nu(t-\tau)}(y,\MTemptyplaceholder)\,\D y}_{\mathbb{H}}^2=  \norm*{BS_{\nu(t-\tau)} \phi}^2 \le \norm*{B}^2 \norm*{S_{\nu(t-\tau)} \phi}^2\\
				&\le \norm*{B}^2\norm*{\phi}^2=\norm*{B}^2.
			\end{align*}
			For the step $n\leadsto n+1$ bound
			\begin{align*}
				K_{n+1} &= \norm*{K_0+ \int_\Lambda\phi(y)\int_\tau^t \int_\Lambda G_{\nu(t-s)}(y,\eta)f'(X_s(\eta))u_{s,\tau}^{(n)}(\eta)\,\D \eta\D s\D y}_{\mathbb{H}}\\
				&\le \norm*{B} + \int_\tau^t \norm*{\int_\Lambda u_{s,\tau}^{(n)}(\eta)f'(X_s(\eta))\int_\Lambda G_{\nu(t-s)}(y,\eta)\phi(y)\,\D y\D\eta}_{\mathbb{H}}\,\D s.
			\end{align*}
			Define $\tilde{\phi}(\MTemptyplaceholder)\coloneqq f'(X_s(\MTemptyplaceholder))\int_\Lambda G_{\nu(t-s)}(y,\MTemptyplaceholder)\phi(y)\,\D y = f'(X_s(\MTemptyplaceholder)) S_{\nu(t-s)}^\ast \phi (\MTemptyplaceholder)$ and note that $\norm{\tilde{\phi}}\le \norm*{f'}_\infty$. Using the induction hypothesis for $t=s$ and $\tilde{\phi}/\norm{f'}_\infty$, we find
			\begin{align*}
				 K_{n+1}&\le \norm*{B} + \norm*{f'}_\infty\int_\tau^t\left(\norm*{B} + \sum_{m=1}^n \frac{\norm*{B}}{m!}\norm*{f'}_\infty^{m}(s-\tau)^{m}\right)\,\D s\\
				 &=\norm*{B} + \norm*{f'}_\infty(t-\tau)\norm*{B} + \sum_{m=1}^n \frac{\norm*{B}}{m!(m+1)}\norm*{f'}_\infty^{m+1}(t-\tau)^{m+1}\\
				 &= \norm*{B} + \sum_{m=1}^{n+1}\frac{\norm*{B}}{m!}\norm*{f'}_\infty^m(t-\tau)^m.\qedhere
			\end{align*}
		\end{proof}
		\begin{lemma}\label{lem:convergence}
			Grant Assumptions \ref{assump:generalA} and \ref{assump:specific_F}. For all $\tau\in[0,T]$, $t\in(\tau,T]$, $\nu>0$ and all realizations $(X_s)_{s\in(\tau,t]}$ we have
			\begin{equation*}
				\sup_{\norm*{\phi}=1}\norm*{\int_\Lambda\phi(y)\mathcal{D}_\tau X_t(y)\,\D y - \int_\Lambda\phi(y)u_{t,\tau}^{(n)}(y)\,\D y}_{\mathbb{H}}\xrightarrow{n\to\infty}0.
			\end{equation*}
		\end{lemma}
		\begin{proof}
			Using the Fubini Theorem for Bochner integrals (Equation (1.9) of \cite{DaPrato2014}), we find for any $\phi\in L^2(\Lambda)$ with $\norm*{\phi}=1$ that
			\begin{align*}
				K_n(t,\phi)&:= \norm*{\int_\Lambda\phi(y)\mathcal{D}_\tau X_t(y)\,\D y - \int_\Lambda\phi(y)u_{t,\tau}^{(n)}(y)\,\D y}_{\mathbb{H}}\\
&= \norm*{\int_\Lambda \phi(y)\int_\tau^t\int_\Lambda G_{\nu(t-s)}(y,\eta)f'(X_s(\eta))[\mathcal{D}_\tau X_s(\eta)- u_{s,\tau}^{(n-1)}(\eta)]\,\D\eta\D s\D y}_\mathbb{H}\\
				&\le\int_\tau^t \norm*{\int_\Lambda [\mathcal{D}_\tau X_s(\eta)- u_{s,\tau}^{(n-1)}(\eta)]f'(X_s(\eta))\int_\Lambda G_{\nu(t-s)}(y,\eta)\phi(y)\,\D y\D\eta}_{\mathbb{H}}\D s.
			\end{align*}			
			Define $\tilde{\phi}(\MTemptyplaceholder)\coloneqq f'(X_s(\MTemptyplaceholder))\int_\Lambda G_{\nu(t-s)}(y,\MTemptyplaceholder)\phi(y)\,\D y = f'(X_s(\MTemptyplaceholder))S_{\nu(t-s)}^\ast\phi(\MTemptyplaceholder)$ and note that $\norm*{\tilde{\phi}}\le \norm*{f'}_\infty$. Then we see that
			\begin{align*}
				K_n(t,\phi)&\le \norm*{f'}_\infty\int_\tau^t \norm*{\int_\Lambda\frac{\tilde{\phi}(\eta)}{\norm*{\tilde{\phi}}}[\mathcal{D}_\tau X_s(\eta)- u_{s,\tau}^{(n-1)}(\eta)]\,\D\eta}_{\mathbb{H}}\D s\\
				&\le \norm*{f'}_\infty\int_\tau^t \sup_{\norm*{\phi}=1}K_{n-1}(s,\phi)\,\D s.
			\end{align*}
			Gronwall's Lemma in the specific form of Lemma 6.2 of \cite{SanzSole2005} with $k_1=k_2=0$, applied to $\sup_{\norm*{\phi}=1}K_{n}(t,\phi)$, yields the claim.
		\end{proof}
		\begin{proof}[Proof of Lemma \ref{lem:boundedness_convergence}]
			The claim is trivial for $t\le\tau$. Hence assume that $t>\tau$. Lemma \ref{lem:convergence} shows that $\int_\Lambda \phi(y)u_{t,\tau}^{(n)}(y)\,\D y$ converges in $\mathbb{H}$ to $\int_\Lambda \phi(y)\mathcal{D}_\tau X_t(y)\,\D y$ for any $\phi\in L^2(\Lambda)$, $\tau\in[0,T]$ and $t\in(\tau,T]$. We conclude by using the triangle inequality with respect to $\norm*{}_{\mathbb{H}}$ and Lemma \ref{lem:Picard2}.
		\end{proof}
		\begin{proof}[Proof of Lemma \ref{lem:Variance_Bound}]
		The Poincaré inequality gives $\var(\mathcal{I}_\nu)\le \EV*{\norm*{\mathcal{D}\mathcal{I}_\nu}_\mathfrak{H}^2}$.
We compute, using Lemma \ref{lem:boundedness_convergence}:
\begin{align*}
		 \EV*{\norm*{\mathcal{D}\mathcal{I}_\nu}_\mathfrak{H}^2}&= 4\EV*{\int_0^T\norm*{\int_\tau^T\int_\Lambda \frac{f(X_t(\eta))f'(X_t(\eta))}{\sigma(\eta)^2}\mathcal{D}_{\tau}X_t(\eta)\,\D\eta \D t}_{\mathbb{H}}^2\D \tau}\\
		 &\le 4\EV*{\int_0^T(T-\tau)\int_\tau^T\norm*{\int_\Lambda \frac{f(X_t(\eta))f'(X_t(\eta))}{\sigma(\eta)^2}\mathcal{D}_{\tau}X_t(\eta)\,\D\eta}_{\mathbb{H}}^2\,\D t\D \tau}\\
		 &\le 4\norm*{B}^2e^{2\norm*{f'}_\infty T}\EV*{\int_0^T (T-\tau)\int_\tau^T\norm*{\frac{f(X_t)f'(X_t)}{\sigma^2}}^2\,\D t\,\D \tau}\\
		 &\le 4\norm*{B}^2e^{2\norm*{f'}_\infty T}T^2\norm*{B\inv}^2\norm*{f'}_\infty^2 \EV*{\mathcal{I}_\nu}.
	\end{align*}
	Applying the Chebyshev inequality and recalling $\EV*{\mathcal{I}_\nu}\to\infty$ completes the proof.
	\end{proof}
	
	\subsubsection{Da Prato-Debussche trick}
	
	We start by controlling the linear part $\bar{\mathcal{I}}_{t,\nu}$.

	\begin{lemma}\label{lem:variance_linear}
		Grant Assumptions \ref{assump:general} and \ref{assump:generalA}. Then $\var(\bar{\mathcal{I}}_{t,\nu})\lesssim \EV*{\bar{\mathcal{I}}_{t,\nu}}$ holds for any $t\in[0,T]$.
	\end{lemma}

\begin{proof}
	Having the Poincaré inequality at hand, we can use it even for the linear (Gaussian) case, where explicit calculation are possible. We find the bound
		\begin{align*}
			\var\left(\int_0^t\norm*{B\inv \bar{X}_s}^2\,\D s\right)&\le \int_0^t \EV*{\norm*{2\int_\tau^t\int_\Lambda \bar{X}_s(\eta)\sigma(\eta)^{-2}G_{\nu(s-\tau)}(\eta,\MTemptyplaceholder)\,\D \eta\,\D s}_{\mathbb{H}}^2}\,\D\tau\\
			&\le 4t\int_0^t \int_\tau^t \EV*{\norm*{S_{\nu(s-\tau)}[\bar{X}_s/\sigma^2]}_{\mathbb{H}}^2}\,\D s\D\tau\\
			&\le 4t\norm*{B}^2\norm*{B\inv}^2\int_0^t\int_0^t\norm*{B\inv \bar{X}_t}^2\,\D \tau\D s\\
			&\le 4t^2 \norm*{B}^2\norm*{B\inv}^2\EV*{\bar{\mathcal{I}}_{t,\nu}}.\qedhere
		\end{align*}
	\end{proof}
		We proceed by proving a path-wise bound on the growth of the non-linear part $(\tilde{X}_t)_{t\in[0,T]}$ with respect to the linear part $(\bar{X}_t)_{t\in[0,T]}$.
		\begin{lemma}\label{lem:controlnon-linear}
			Under Assumptions \ref{assump:general} and \ref{assump:general_F} we have for $t\in[0,T]$
			\begin{equation*}
				\norm{\tilde{X}_t}^2 \le 3 e^{2D\theta t}\left(\norm*{S_{\nu t}X_0}^2 + t(D\theta )^2\left( \int_0^t \norm{\bar{X}_s}^2\,\D s + t\right)\right),
			\end{equation*}
			in particular $\EV{\norm{\tilde{X}_t}^2} \lesssim 1 + \phi(\nu,t)$.
		\end{lemma}
		\begin{proof}
			Note that we have
			\begin{align*}
				\norm{\tilde{X}_t}&=\norm*{S_{\nu t}X_0 + \theta\int_0^t S_{\nu(t-s)} F(\tilde{X}_s+\bar{X}_s)\,\D s}\\
				&\le \norm*{S_{\nu t}X_0} + \theta \int_0^t \norm*{F(\tilde{X}_s+\tilde{X}_s)}\,\D s\\
				&\le \norm*{S_{\nu t}X_0} + D\theta \left(\int_0^t \norm*{\bar{X}_s}\,\D s +  t + \int_0^t\norm{\tilde{X}_s}\,\D s\right).
			\end{align*}
			Gronwall's inequality yields
			\begin{equation*}
				\norm{\tilde{X}_t}\le e^{D\theta t}\left(\norm*{S_{\nu t}X_0} + D\theta\left(\int_0^t\norm*{\bar{X}_s}\,\D s + t\right)\right)
			\end{equation*}
			and consequently
			\begin{equation*}
				\norm{\tilde{X}_t}^2 \le 3 e^{2D\theta t}\left(\norm*{S_{\nu t}X_0}^2 + (D\theta )^2\left( t\int_0^t \norm*{\bar{X}_s}^2\,\D s + t^2\right)\right).
			\end{equation*}
			The second claim follows from noting that the initial condition is negligible due to $\EV*{\norm*{S_{\nu t}X_0}^2}\le \EV*{\norm*{X_0}^2}<\infty$ for all $\nu,t>0$ and applying \eqref{eq:rate_linear} together with the monotonicity of $t\mapsto\phi(\nu,t)$.
		\end{proof}
		\begin{proof}[Proof of Lemma \ref{lem:non-linearexpectation}]
		We split the proof into two parts:
		\begin{enumerate}[(i)]
			\item For any $t\in [0,T]$ we have the bound $\EV*{\norm{B\inv F(X_t)}^2} \lesssim 1 + \phi(\nu,t)$ and
			\item For any $t\in [0,T]$ we have the bound
			\begin{align*}
				\EV*{\norm{B\inv F(X_t)}^2}\ge \hspace{-8em}&\\
				&\left[\phi(\nu,t)\left(\norm*{B\inv}^{-2}/2 - 3\norm*{B}^2(D\theta )^2e^{2D\theta t}t^2\right) - 3e^{2D\theta t}\left(\EV*{\norm*{X_0}^2}+(D\theta t)^2\right)\right]\vee 0.
			\end{align*}
			\end{enumerate}
	
			First we prove the upper bound (i): Assumption \ref{assump:general_F} and Lemma \ref{lem:controlnon-linear} together with \eqref{eq:linearpart_secondmoment}, \eqref{eq:rate_linear} yield the bound
				\begin{align*}
				\EV*{\norm*{B\inv F(X_t)}^2}&=\EV*{\norm*{B\inv F(\bar{X}_t + \tilde{X}_t)}^2}\\
				&\le (\norm*{B\inv} D)^2\left(\EV{\norm{\bar{X}_t+\tilde{X}_t}^2} + 1\right)\\
				&\le (\norm*{B\inv}D)^2 \left(2\EV{\norm{\bar{X}_t}^2}+2\EV{\norm{\tilde{X}_t}^2} + 1\right)\\
				&\lesssim \phi(\nu,t) +1.
			\end{align*}		
			We proceed to the lower bound (ii). We use \eqref{eq:linearpart_secondmoment}, Proposition \ref{prop:LowerboundNorm} with $\alpha=1/2$ and Lemma \ref{lem:controlnon-linear} to obtain the bound
			\begin{align*}
				\norm*{B}^2 d^{-2}\EV*{\norm*{B\inv F(X_t)}^2}\hspace{-10em}&\\
				&\ge \EV*{\norm{\bar{X}_t}^2}/2-\EV*{\norm{\tilde{X}_t}^2}\\
				&\ge\EV*{\norm{\bar{X}_t}^2}/2 -  3 e^{2D\theta t}\left(\EV*{\norm*{X_0}^2} + t(D\theta )^2\left(\int_0^t\EV*{\norm{\bar{X}_s}^2}\,\D s + t\right)\right)\\
				&\begin{multlined}[t]\ge \norm*{B\inv}^{-2}\phi(\nu,t)/2-
					3 e^{2D\theta t}\left(\EV*{\norm*{X_0}^2} + t(D\theta )^2\left(\norm*{B}^{2}\int_0^t\phi(\nu,s)\,\D s + t\right)\right)\end{multlined}\\
				&\ge
				\begin{multlined}[t]
					\phi(\nu,t) \left[\norm*{B\inv}^{-2}/2 - 3(D\theta )^2\norm*{B}^{2} e^{2D\theta  t}t^2\right] -
					 3e^{2D\theta  t}\left(\EV*{\norm*{X_0}^2} + (D\theta t)^2\right).
					\end{multlined}
			\end{align*}
			To conclude, we fix  $\tau\in (0,T]$ such that $c\coloneqq \norm*{B\inv}^{-2}/2 - 3\norm*{B}^2(D\theta )^2e^{2D\theta  \tau}\tau^2>0$. Using \eqref{eq:compatibility} and \eqref{eq:rate_linear}, we obtain
			\begin{align*}
				\EV*{\mathcal{I}_\nu}&\ge \int_0^\tau \EV*{\norm*{B\inv F(X_t)}^2}\,\D t\ge c\int_0^\tau \phi(\nu,t)\,\D t - 3e^{2D\theta t}\left(\EV*{\norm*{X_0}^2}+(D\theta t)^2\right)\\
				&\gtrsim \int_0^\tau \phi(\nu,t)\,\D t\ge (\tau/T)^2 \int_0^T \phi(\nu,t)\,\D t\gtrsim \EV*{\bar{\mathcal{I}}_{T,\nu}}.\qedhere
			\end{align*}
\end{proof}
		
\pagebreak
\subsection{Proofs for Subsection \ref{subsec:GeneralRates}}
		
\begin{lemma}\label{lem:inverse_Fisher}
		Grant Assumptions \ref{assump:general}, \ref{assump:generalA} and \ref{assump:general_F}. Then we have
		\begin{equation*}
			\frac{\EV*{\mathcal{I}_{T,\nu}}}{\mathcal{I}_{T,\nu}} = \mathcal{O}_{\prob*{}}(1).
		\end{equation*}
	\end{lemma}
		\begin{proof}
			 Note that Lemma \ref{lem:controlnon-linear} implies  for every $\tilde{t}\in[0,T]$ ($\omega$-wise)
			\begin{align}
				\int_0^{\tilde{t}} \norm{\tilde{X}_t}^2\,\D t &\le 3 \int_0^{\tilde{t}}e^{2D\theta t}\left(\norm*{S_{\nu t}X_0}^2 + t^2(D\theta )^2\right)\,\D t + 3(\norm*{B}D\theta)^2 \int_0^{\tilde{t}}e^{2D\theta t}t\,\D t\bar{\mathcal{I}}_{\tilde{t},\nu}\nonumber\\
				&=\colon u_1(\tilde{t}) + u_2(\tilde{t})\bar{\mathcal{I}}_{\tilde{t},\nu}.\label{eq:upperboundwide}
			\end{align}
			Now choose  $c\in (0,2^{-1})$ and $\tau\in(0,T]$ such that $u_2(\tau)<2^{-1} - c$, which is always possible by the monotonicity of $u_2$. Then we can use the lower bound for $F$ of Assumption \ref{assump:general_F}, Proposition \ref{prop:LowerboundNorm} with $\alpha=1/2$ and \eqref{eq:upperboundwide} to find
			\begin{align*}
				\mathcal{I}_\nu&\ge \int_0^\tau \norm*{B\inv F(X_t)}^2\,\D t \ge \norm*{B}^{-2}d^2\int_0^\tau \norm{\bar{X}_t+\tilde{X}_t}^2\,\D t\\
				&\ge \norm*{B}^{-2}d^2\left(2^{-1}\bar{\mathcal{I}}_{\tau,\nu} - \int_0^\tau \norm{\tilde{X}_t}^2\,\D t\right)\ge \norm*{B}^{-2}d^2\left(\left(2^{-1}-u_2(\tau)\right)\bar{\mathcal{I}}_{\tau,\nu}-u_1(\tau)\right).
			\end{align*}
			Introducing $Y_{\tau,\nu}\coloneqq \bar{\mathcal{I}}_{\tau,\nu} - u_1(\tau)/(2^{-1}-u_2(\tau))$, this implies that
			\begin{align*}
				\prob*{\mathcal{I}_{\nu}\le \frac{cd^2}{\norm*{B}^2}\EV*{\bar{\mathcal{I}}_{\tau,\nu}}}&\le \prob*{\left(2^{-1}-u_2(\tau)\right)\bar{\mathcal{I}}_{\tau,\nu}-u_1(\tau)\le c\EV*{\bar{\mathcal{I}}_{\tau,\nu}}}\\
				&=\prob*{Y_{\tau,\nu}- \EV*{Y_{\tau, \nu}}\le \frac{(u_2(\tau)+c-2^{-1})\EV*{\bar{\mathcal{I}}_{\tau,\nu}} + \EV*{u_1(\tau)}}{2^{-1}-u_2(\tau)}}.
			\end{align*}
			For $\nu>0$ small enough, the right-hand side is positive by our choice of $\tau$ and $c$, as well as $\EV*{u_1(\tau)}<\infty$. Consequently, we find
			\begin{equation*}
				\prob*{\mathcal{I}_{\nu}\le \frac{cd^2}{\norm*{B}^2}\EV*{\bar{\mathcal{I}}_{\tau,\nu}}}\le\prob*{\abs*{Y_{\tau,\nu} - \EV*{Y_{\tau,\nu}}}\ge \frac{(u_2(\tau)+c-2^{-1})\EV*{\bar{\mathcal{I}}_{\tau,\nu}} + \EV*{u_1(\tau)}}{-2^{-1}+u_2(\tau)}}.
			\end{equation*}
			Chebyshev's inequality yields
			\begin{equation*}
				\prob*{\mathcal{I}_{\nu}\le \frac{cd^2}{\norm*{B}^2}\EV*{\bar{\mathcal{I}}_{\tau,\nu}}}\le \var\left(Y_{\tau,\nu}\right)\left(\frac{(u_2(\tau)+c-2^{-1})\EV*{\bar{\mathcal{I}}_{\tau,\nu}} + \EV*{u_1(\tau)}}{-2^{-1}+u_2(\tau)}\right)^{-2}.
			\end{equation*}
			By Lemma \ref{lem:variance_linear} and $\var(u_1(\tau)/(2^{-1}-u_2(\tau)))<\infty$, the right-hand side vanishes as $\nu\to 0$. Combining \eqref{eq:rate_linear}, \eqref{eq:compatibility} and Lemma \ref{lem:non-linearexpectation} yields
			\begin{equation*}
				\EV*{\bar{\mathcal{I}}_{\tau,\nu}}\gtrsim \int_0^\tau \phi(\nu,s)\,\D s\gtrsim \int_0^T \phi(\nu,s)\,\D s\gtrsim \EV*{\bar{\mathcal{I}}_{T,\nu}}\gtrsim \EV*{\mathcal{I}_\nu}.
			\end{equation*}
			Hence, for $\epsilon=\liminf_{\nu\to 0}cd^2\norm{B}^{-2}\EV*{\bar{\mathcal{I}}_{\tau,\nu}}/\EV*{\mathcal{I}_\nu}>0$ we have proved the convergence $\prob{\mathcal{I}_\nu\EV*{\mathcal{I}_\nu}^{-1}\le \epsilon}\to 0$ as $\nu\to0$.
\end{proof}

\begin{proof}[Proof of Proposition \ref{thm:non-linearrate}]
Consider the decomposition \eqref{eq:decomthetanu}.
Lemma \ref{lem:non-linearexpectation} combined with \eqref{eq:rate_linear} proves $\EV*{\mathcal{I}_{\nu}}^{1/2} \sim \phi(\nu)^{1/2}$. Furthermore, we have $\mathcal{M}_{\nu}/\EV*{\mathcal{I}_{\nu}}^{1/2}=\mathcal{O}_{\prob*{}}(1)$ by the definition of the quadratic variation. Finally, $\frac{\EV{\mathcal{I}_{\nu}}}{\mathcal{I}_{\nu}} = \mathcal{O}_{\prob*{}}(1)$ follows by Lemma \ref{lem:inverse_Fisher}, which completes the proof.
\end{proof}
	
	\begin{proof}[Proof of Proposition \ref{lem:general_rate}]
		We proceed as in the proof of Proposition \ref{thm:non-linearrate}. Note that $(S_t)_{t\ge 0}$ and $B$ commute and therefore
		\begin{equation*}
			\EV*{\norm*{\bar{X}_t}_B^2}=\EV*{\norm*{B\inv \bar{X}_t}^2}= \int_0^t \norm*{B\inv S_{\nu(t-s)} B}_{\operatorname{HS}}^2\,\D s=\phi(\nu,t).
		\end{equation*}
		For the non-linear part we find similarly to Lemma \ref{lem:controlnon-linear} that
		\begin{equation*}
			\norm{\tilde{X}_t}_B \le \norm*{X_0}_B + D\theta\left(\int_0^t\norm{\bar{X}_s}_B\,\D s+t+\int_0^t\norm{\bar{X}_s}_B\,\D s\right).
		\end{equation*}
		Using Gronwall's inequality, we arrive at
		\begin{equation*}
			\norm{\tilde{X}_t}_B^2 \le 3 e^{2D\theta t}\left(\norm*{X_0}_B^2 + t(D\theta )^2\left( \int_0^t \norm{\bar{X}_s}_B^2\,\D s + t\right)\right).
		\end{equation*}
Working with $\norm*{}_B$ instead of $\norm*{}$ in Lemma \ref{lem:non-linearexpectation}, we deduce $\EV*{\mathcal{I}_\nu}\sim\phi(\nu)$.
		
		We are left with extending Lemma \ref{lem:variance_linear} and thereby Lemma \ref{lem:inverse_Fisher}. To this end set $\bar{X}_{k,t}\coloneqq \iprod*{B\inv \bar{X}_t}{u_k}$ and use Wick's formula to find
		\begin{align*}
			\var\left(\bar{\mathcal{I}}_{T,\nu}\right)&=\int_0^T\int_0^T \cov\left(\norm*{B\inv \bar{X}_t}^2,\norm*{B\inv \bar{X}_s}^2\right)\,\D s\D t\\
			&=2\int_0^T\int_0^t \sum_{k\in\mathbb{N}}\cov\left(\bar{X}_{k,t}^2,\bar{X}_{k,s}^2\right)\,\D s\D t\\
			&=4\int_0^T\int_0^t \sum_{k\in\mathbb{N}}\cov\left(\bar{X}_{k,t},\bar{X}_{k,s}\right)^2\,\D s\D t\\
			&=4\int_0^T\int_0^t\sum_{k\in\mathbb{N}}\left(\int_0^s \iprod*{B\inv S_{\nu(t-z)}Be_k}{B\inv S_{\nu(s-z)}Be_k}\,\D z\right)^2\,\D s\D t\\
			&\le 4\int_0^T\int_0^ts\int_0^s \sum_{k\in\mathbb{N}}\norm*{S_{\nu(s-z)}e_k}^2\,\D z\D s\D t\\
			&\lesssim \EV*{\bar{\mathcal{I}}_{T,\nu}}.\qedhere
		\end{align*}
	\end{proof}

	\section{Proofs for Section \ref{sec:discretize}}\label{sec:Obs_schemes_proofs}
		
	\subsection{\textbf{Localized Perspective}}\label{subsec:Localized_Proof}
	
	\begin{proof}[Proof of Lemma \ref{lem:Sclaing_Linear_Part}]
		Fix some $t\in (t_0-\delta_t,t_0+\delta_t)$. By condition \eqref{eq:assump_nonpara_a}, we find
		\begin{align*}
		\int_{\Lambda_{\delta_y}(y_0)}\EV*{\bar{X}_t(y)^2}\,\D y&= \int_{\Lambda_{\delta_y}(y_0)}\int_0^t\norm*{BG_{\nu(t-s)}(y,\MTemptyplaceholder)}^2\,\D s\D y\\
		&\sim_{c_3\norm*{B\inv}^{-2},c_4\norm*{B}^2}\delta_y^d \phi(\nu,t),
	\end{align*}
	where $\sim_{c_3\norm*{B\inv}^{-2},c_4\norm*{B}^2}$ means $\ge c_3\norm*{B\inv}^{-2} \cdot$ and $\le c_4\norm*{B}^2\cdot$.
	\end{proof}
			
	\begin{proof}[Proof of Lemma \ref{lem:Assumptions_hold_a}] Fix some $t\in (t_0-\delta_t,t_0+\delta_t)$.
	
	\textbf{Step 1:} We start by treating the linear part $\bar{X}$. Using (i), we find for any $y\in\Lambda_{\delta_y}(y_0)$
	\begin{align*}
	\EV*{\bar{X}_t(y)^2} &= \EV*{\Big(\int_0^t \int_\Lambda G_{\nu(t-s)}(y,\eta)\mathcal{W}(\D s,\D\eta)\Big)^2}\nonumber\\
	&= \int_0^t \norm*{BG_{\nu(t-s)}(y,\MTemptyplaceholder)}^2\,\D s \le c_5\norm*{B}^2\int_0^t \norm*{S_{\nu(t-s)}}_{\operatorname{HS}}^2\,\D s\nonumber\\
	&= c_5\norm*{B}^2\phi(\nu,t).
	\end{align*}
	In particular, $\norm{\EV{\bar{X}_t^2}}_\infty\le c_5\norm*{B}^2\phi(\nu,t)$.
	
	\textbf{Step 2:} We treat the non-linear part $\tilde{X}$. Fix any $y\in \Lambda_{\delta_y}(y_0)$ and note that for any $z\in C(\Lambda)$, using non-negativity of $G_t$,
	\begin{align}
		[S_t z](y)^2 &=\left(\int_\Lambda G_t(y,\eta)z(\eta)\,\D \eta\right)^2\le \int_\Lambda G_t(y,\eta)z(\eta)^2\,\D \eta \int_\Lambda G_t(y,\eta)\,\D\eta\nonumber\\
		&=\colon c_8[S_t z^2](y).\label{eq:Riesz_Application}
	\end{align}
	Note that $c_8<\infty$, as (i) and (ii) imply that $G_t(y,\eta)\D\eta$ is a finite measure for every $y\in\Lambda$ due to the Riesz-Markov Theorem (Theorem 2.4 of \cite{Rudin1987}). Combining this with (iv), we arrive at the bound
	\begin{align*}
		\norm{\EV{\tilde{X}_t^2}}_\infty &\le 2\norm*{\EV*{X_0^2}}_\infty + 2t\int_0^t \norm*{\EV*{[S_{\nu(t-s)}\theta F(X_s)]^2}}_\infty\,\D s\\
		&\le 2c_9 + 2tc_8\int_0^t \norm*{S_{\nu(t-s)}\EV*{\theta F(X_s)^2}}_\infty\,\D s\\
		&\le 2c_9 + 4tc_6c_8(\theta D)^2\int_0^t \norm*{\EV{\bar{X}_s^2} + \EV{\tilde{X}_s^2}}_\infty\,\D s + 4t^2c_6c_8(\theta D)^2\\
		&\le 2c_9 + 4tc_6c_8(\theta D)^2\int_0^t \left(c_5\norm*{B}^2\phi(\nu,s) + \norm{\EV{\tilde{X}_s^2}}_\infty\right)\,\D s + 4t^2c_6c_8(\theta D)^2\\
		&\le 2c_9 + 4t^2c_6c_8(\theta D)^2\left(c_5\norm*{B}^2\phi(\nu,t)  + 1\right) + 4Tc_6c_8(\theta D)^2\int_0^t \norm{\EV{\tilde{X}_s^2}}_\infty\,\D s
	\end{align*}
	Gronwall's inequality yields the bound
	\begin{equation*}
		\norm{\EV{\tilde{X}_t^2}}\le e^{4Ttc_6c_8(\theta D)^2}\left(2c_9 + 4t^2c_5c_6c_8 (\norm*{B}\theta D)^2(\phi(\nu,t) + 1)\right).
\end{equation*}
	\textbf{Step 3:}
	Combining the previous estimate with Lemma \ref{lem:Sclaing_Linear_Part}, we find the upper bound
	\begin{align*}
		\int_{\Lambda_{\delta_y}(y_0)} \EV*{X_t(y)^2}\,\D y \hspace{-8em}&\\
		&\le 2\int_{\Lambda_{\delta_y}(y_0)} \EV*{\bar{X}_t(y)^2}\,\D y + 2\int_{\Lambda_{\delta_y}(y_0)} \EV*{\tilde{X}_t(y)^2}\,\D y\\
		&\le 2\delta_y^d \abs*{\Lambda}c_5\norm*{B}^2 \phi(\nu,t) + 2\abs*{\Lambda}\delta_y^d e^{4Ttc_6c_8(\theta D)^2}\left(2c_9 + 4t^2c_5c_6c_8 (\norm*{B}\theta D)^2(\phi(\nu,t) + 1)\right)\\
		&\le
		\begin{multlined}[t]
			2\delta_y^d \phi(\nu,t)\norm*{B}^2\abs*{\Lambda}c_5\left(1 + 4e^{4T tc_6c_8 (\theta D)^2}t^2 c_6c_8(\theta D)^2\right) +\\
			 2\abs*{\Lambda}\delta_y^d e^{4Ttc_6c_8(\theta D)^2}\left(2c_9 + 4t^2c_5c_6c_8 (\norm*{B}\theta D)^2\right).
			\end{multlined}
 	\end{align*}
	Similarly, the lower bound
		\begin{align}
		\int_{\Lambda_{\delta_y}(y_0)} \EV*{X_t(y)^2}\,\D y \hspace{-8em}&\nonumber\\
		&\ge 2^{-1}\int_{\Lambda_{\delta_y}(y_0)} \EV*{\bar{X}_t(y)^2}\,\D y -\int_{\Lambda_{\delta_y}(y_0)} \EV*{\tilde{X}_t(y)^2}\,\D y\nonumber\\
		&\ge c_3 \norm*{B}^2 \delta_y^d \phi(\nu,t)/2 - e^{4Ttc_6c_8(\theta D)^2}\left(2c_9 + 4t^2c_5c_6c_8 (\norm*{B}\theta D)^2(\phi(\nu,t) + 1)\right)\nonumber\\
		\begin{split}&\ge
		\begin{multlined}[t]
			\delta_y^d \phi(\nu,t)\norm*{B}^2\left(c_3/2 - 4\abs*{\Lambda}e^{4Ttc_6c_8(\theta D)^2}c_5c_6c_8(\theta D)^2t^2\right)-\\
			\abs*{\Lambda}\delta_y^d e^{4Ttc_6c_8(\theta D)^2}\left(2c_9 + 4t^2c_5c_6c_8 (\norm*{B}\theta D)^2\right)
			\end{multlined}
			\end{split}\label{eq:largeenough}
	\end{align}
	follows. By our choice of $t_0$ in \eqref{item:assump_nonpara_t}, we see that  Assumption \ref{assump:Non_para_too_strong} holds.
	\end{proof}
	\begin{proof}[Proof of Lemma \ref{lem:Assumptions_hold_b}]
The main ingredient is that the identity $F(-X_s)=-F(X_s)$ in the variation-of-constants formula implies that $-X_t$ is a solution to the same SPDE with cylindrical Brownian motion $(-W_t)$ and initial condition $-X_0=0$. By weak uniqueness, we have $-X_t\stackrel{d}{=}X_t$ and thus $\EV{X_t(y)}=\EV{\tilde X_t(y)}=0$.

Fix some $t\in (t_0-\delta_t,t_0+\delta_t)$. By the Poincaré inequality and Lemma \ref{lem:boundedness_convergence} we therefore find for any $t\in[0,T]$ and $y\in\Lambda$ the bound
	\begin{align*}
		\EV*{\tilde{X}_t(y)^2}&\le \EV*{\norm*{\mathcal{D}\tilde{X}_t(y)}_{\mathfrak{H}}^2}\\
		&=\int_0^T\EV*{ \norm*{\int_0^t \int_\Lambda G_{\nu(t-s)}(y,\eta)f'(X_s(\eta))\mathcal{D}_\tau X_s(\eta)\,\D\eta\D s}_{\mathbb{H}}^2}\D \tau\\
		&\le \int_0^t t\int_0^t \EV*{\norm*{\int_\Lambda G_{\nu(t-s)}(y,\eta)f'(X_s(\eta))\mathcal{D}_\tau X_s(\eta)\,\D\eta}_{\mathbb{H}}^2}\,\D s\D \tau\\
		&\le \norm*{B}^2e^{2\norm*{f'}_\infty t}\int_0^t t\int_0^t\EV*{\norm*{G_{\nu(t-s)}(y,\MTemptyplaceholder) f'(X_s(\MTemptyplaceholder))}^2}\,\D s\D \tau\\
		&\le \norm*{B}^2e^{2\norm*{f'}_\infty t}\norm*{f'}_\infty^2 t^2 \int_0^t \norm*{G_{\nu(t-s)}(y,\MTemptyplaceholder)}^2\,\D s.
	\end{align*}
Condition \eqref{eq:assump_nonpara_a} and Lemma \ref{lem:Sclaing_Linear_Part} yield the upper bound
	\begin{align*}
		\int_{\Lambda_{\delta_y}(y_0)} \EV*{X_t(y)^2}\,\D y \hspace{-5em}&\\
		&\le 2\int_{\Lambda_{\delta_y}(y_0)} \EV*{\bar{X}_t(y)^2}\,\D y + 2\int_{\Lambda_{\delta_y}(y_0)} \EV*{\tilde{X}_t(y)^2}\,\D y\\
		&\le 2\norm*{B}^2\Big(c_4\delta_y^d \phi(\nu,t) + e^{2\norm*{f'}_\infty t}\norm*{f'}_\infty^2 t^2 \int_0^t \int_{\Lambda_{\delta_y}(y_0)}\norm*{G_{\nu(t-s)}(y,\MTemptyplaceholder)}^2\,\D y\D s\Big)\\
			&\le 2\delta_y^d \phi(\nu,t)c_4\norm*{B}^2\left(1 +e^{2\norm*{f'}_\infty t}\norm*{f'}_\infty^2t^2\right).
	\end{align*}
	Similarly, we find the lower bound
		\begin{align*}
		\int_{\Lambda_{\delta_y}(y_0)} \EV*{X_t(y)^2}\,\D y \hspace{-7em}&\\
		&\ge 2^{-1}\int_{\Lambda_{\delta_y}(y_0)} \EV*{\bar{X}_t(y)^2}\,\D y -\int_{\Lambda_{\delta_y}(y_0)} \EV*{\tilde{X}_t(y)^2}\,\D y\\
		&\ge
			\norm*{B\inv}^{-2}\Big(c_3 \delta_y^d \phi(\nu,t)/2 - e^{2\norm*{f'}_\infty t}\norm*{f'}_\infty^2 t^2 \int_0^t \int_{\Lambda_{\delta_y}(y_0)}\norm*{G_{\nu(t-s)}(y,\MTemptyplaceholder)}^2\,\D y\D s\Big)\\
			&\ge \delta_y^d \phi(\nu,t)\norm*{B\inv}^{-2}\left(c_3/2 - c_4 e^{2\norm*{f'}_\infty t}\norm*{f'}_\infty^2 t^2\right).
	\end{align*}
	Assumption \ref{assump:Non_para_too_strong} holds by our choice of $t_0$ in \eqref{item:assump_nonpara_t_2}.
	\end{proof}
	\begin{proof}[Proof of  Lemma \ref{lem:Assumptions_hold_Example}]
	We show that Condition \eqref{eq:assump_nonpara_a} and the conditions (i) - (iv) of Lemma \ref{lem:Assumptions_hold_a}  and (i) - (iii) of Lemma \ref{lem:Assumptions_hold_b} hold for the fractional heat equation \eqref{eq:Guiding_Example}.
	
	\textbf{Step 1:} We start by showing Condition \eqref{eq:assump_nonpara_a}. Fix some $t\in (t_0-\delta_t,t_0+\delta_t)$.
	
	From the representation \eqref{eq:fractional_kernel} of $G_t$ and the uniform boundedness of the eigenfunctions $(e_k)_{k\in\mathbb{N}}$ in supremum norm, we get the upper bound
	\begin{equation*}
		\int_{\Lambda_{\delta_y}(y_0)}\norm*{G_t(y_0+\delta_y y,\MTemptyplaceholder)}^2\,\D y \le \sum_{k\in\mathbb{N}}e^{-2\lambda_k^{\alpha/2}t}\norm*{e_k}_\infty^2 \abs*{\Lambda_{\delta_y}(y_0)}\le 2\abs*{\Lambda}\delta_y^d\norm*{S_t}_{\operatorname{HS}}^2/l
	\end{equation*}
	and hence $c_4=2$.
	
	In order to prove the lower bound, recall that $e_k(y) = \sqrt{2/l}\sin(k\pi y/l)$ and $\Lambda_{\delta_y}(y_0) = ((1-\delta_y)y_0,(1-\delta_y)y_0 + \delta_y l]$. Therefore,
	\begin{equation*}
		\int_{\Lambda_{\delta_y}(y_0)}e_k(y)^2\,\D y = \delta_y+\frac{\sin(2\pi k((1-\delta_y)y_0 + \delta_y l)/l)-\sin(2\pi k(1-\delta_y)y_0/l)}{2\pi k}
	\end{equation*}
	and for high frequencies $k>1/(\pi \delta_y(1-\zeta))=\colon k_{\delta_y}$ we can lower bound the integral by $\delta_y\zeta$ for any fixed $\zeta\in (0,1)$.
	
	In order to treat the low frequencies $k\le k_{\delta_y}$, we take for the moment the following claim for granted:
	
	\textbf{Claim:} There exists a $\xi>0$ such that at least for every second $k$ we have $\sin(\pi k y_0/l)^2>\xi$.
	
Without loss of generality we choose the odd numbers $k\in 2\mathbb{N}-1$. By choosing bounded frequencies, we can find a non-empty neighbourhood $U\subset \Lambda$ of $y_0$ such that $\sin(\pi k(y_0+\delta_y(y-y_0))/l)>\xi/2$ on $U$ for all $\delta_y>0$ and all $k\le k_{\delta_y}= 1/(\pi \delta_y(1-\zeta))$, $k\in 2\mathbb{N}-1$. Hence,
	\begin{equation*}
		\int_\Lambda e_k(y_0+\delta_y(y-y_0)))^2\,\D y\ge \abs*{U} \xi/2
	\end{equation*}
	for at least every second $k\le k_{\delta_y}= 1/(\pi\delta_y(1-\zeta))$.
	Together, we find
	\begin{align*}
		\int_\Lambda\norm*{G_t(y_0+\delta_y(y-y_0),\MTemptyplaceholder)}^2\,\D y &=\sum_{k\in\mathbb{N}}e^{-2\lambda_k^{\alpha/2}t}\int_\Lambda e_k(y_0+\delta_y(y-y_0))^2\,\D y\\
		&\ge \sum_{k>k_{\delta_y}}e^{-2\lambda_k^{\alpha/2}t}\zeta + \sum_{k\le k_{\delta_y},k\in 2\mathbb{N}-1}e^{-2\lambda_k^{\alpha/2}t}\abs*{U}\xi/2\\
		&\ge \sum_{k>k_{\delta_y}}e^{-2\lambda_k^{\alpha/2}t}\zeta + \frac{1}{2}\sum_{k\le k_{\delta_y}}e^{-2\lambda_k^{\alpha/2}t}\abs*{U}\xi/2\\
		&\ge \min(\zeta , \abs*{U}\xi/4)\norm*{S_t}_{\operatorname{HS}}^2,
	\end{align*}
	such that
	\begin{equation}
		c_3=\min(\zeta, \abs*{U}\xi/4).\label{eq:c_3}
	\end{equation}	
	\textbf{Proof of the Claim:} Fix some $0<\xi_1< \sin(\pi y_0/l)^2/2$ (note that $\sin(\pi y_0/l)^2>0$ as $y_0<l$) and assume that for some $k\in\mathbb{N}$ we have $e_k(y_0)^2<\xi_1$. Then, by Proposition \ref{prop:LowerboundNorm},
	\begin{align*}
		l/2\abs*{e_{k+1}(y_0)}^2&=\abs*{\sin((k+1)\pi y_0/l)}^2\\
		&=\abs*{\cos(k\pi y_0/l)\sin(\pi y_0/l) + \sin(k\pi y_0/l)\cos(\pi y_0/l)}^2\\
		&\ge \cos(k\pi y_0/l)^2\sin(\pi y_0/l)^2/2 - \sin(k\pi y_0/l)^2\cos(\pi y_0/l)^2\\
		&=\sin(\pi y_0/l)^2\left[(1-e_k(y_0)^2)/2 + e_k(y_0)^2\right]-e_k(y_0)^2\\
		&\ge \frac{1}{2}\sin(\pi y_0/l)^2-\xi_1=\colon \xi_2.
	\end{align*}
	Hence, at least for every second $k\in\mathbb{N}$ we have $e_k(y_0)^2\ge \max(\xi_1,2\xi_2/l)>0$.

\textbf{Step 2:} We show that (i) - (iv) of Lemma \ref{lem:Assumptions_hold_a} hold. Note that $(-\Updelta)^{\alpha/2}$, $\alpha\in (1,2]$ is the generator of a (subordinated) stopped Brownian motion $(Z_t^\alpha)_{t\ge 0}$ \cite{Lischke2020}. In particular, positivity of the Green function holds. Condition (i) follows from \eqref{eq:frac_Laplace_kernelnorm} with $c_5=2/l$. For (ii), note that for any $t\ge 0$, $z\in L^\infty(\Lambda)$, we can bound
	\begin{equation*}
		\norm*{S_tz}_\infty = \sup_{x\in\Lambda}\abs*{\EV*{z(Z_t^\alpha + x)}}\le \norm*{z}_\infty
	\end{equation*}
	and hence (ii) of Lemma \ref{lem:Assumptions_hold_a} holds with $c_6=1$. (iii) is satisfied by the choice of $t_0$ and (iv) by $X_0=0$ with $c_9=0$.
	
\textbf{Step 3:} (i) - (iii) of Lemma \ref{lem:Assumptions_hold_a} hold by the choice of $t_0$, the anti-symmetry assumption on $f$ and $X_0=0$.
\end{proof}

	\begin{proof}[Proof of Proposition \ref{prop:CLT_Localized}]
		Using the Poincaré inequality and Lemma \ref{lem:boundedness_convergence}, we find
	\begin{align*}
		\var\left(\mathcal{I}_{\nu,\delta_t,\delta_y}\right)\hspace{-4em}&\\
		&\le \EV*{\norm*{\mathcal{D}\mathcal{I}_{\nu,\delta_t,\delta_y}}_{\mathbb{H}}^2}\\
		&= 4\EV*{\int_0^T\norm*{\int_{t_0-\delta_t}^{t_0+\delta_t}\int_{\Lambda_{\delta_y}(y_0)} \frac{f(X_s(\eta))f'(X_s(\eta))}{\sigma(\eta)^2}\mathcal{D}_{\tau}X_s(\eta)\,\D\eta \D s}_{\mathbb{H}}^2\D \tau}\\
		 &\le 4\mathrm{E}\left[\int_0^{t_0+\delta_t}2\delta_t\int_{t_0-\delta_t}^{t_0+\delta_t}\norm*{\int_{\Lambda_{\delta_y}(y_0)} \frac{f(X_s(\eta))f'(X_s(\eta))}{\sigma(\eta)^2}\mathcal{D}_{\tau}X_s(\eta)\,\D\eta}_{\mathbb{H}}^2\,\D s\D \tau\right]\\
		 &\le8\norm*{B}^2e^{2\norm*{f'}_\infty T}\mathrm{E}\left[\int_0^{t_0+\delta_t} \delta_t \int_{t_0-\delta_t}^{t_0+\delta_t}\int_{\Lambda_{\delta_y}(y_0)}\frac{f(X_s(y))^2f'(X_s(y))^2}{\sigma(y)^4}\,\D y\D s\,\D \tau\right]\\
		 &\le 8\norm*{B}^2e^{2\norm*{f'}_\infty (t+\delta_t)}\norm*{B\inv}^2\norm*{f'}_\infty^2 (t_0+\delta_t)\delta_t \EV*{\mathcal{I}_{\nu,\delta_t,\delta_y}}\\
		 &\lesssim\EV*{\mathcal{I}_{\nu,\delta_t,\delta_y}}.
	\end{align*}
	This proves the first part of the proposition.
	 Assumption \ref{assump:Non_para_too_strong} implies the bounds
	\begin{equation*}
		c_1\delta_y^d\delta_t \phi(\nu,t_0-\delta_t)\lesssim \int_{t_0-\delta_t}^{t_0+\delta_t}\EV*{\norm*{X_t}_{\delta_y}^2}\,\D t \lesssim c_2\delta_y^d \delta_t\phi(\nu,t_0+\delta_t).
	\end{equation*}
	Since $c_1>0$ and using the assumptions on $B$ and $F$ from Assumption \ref{assump:general_F}, we obtain
	\begin{equation*}
		\EV*{\mathcal{I}_{\nu,\delta_t,\delta_y}}=\int_{t_0-\delta_t}^{t_0+\delta_t}\EV*{\norm*{B\inv F(X_t)}_{\delta_y}^2}\,\D t \sim \delta_y^d\delta_t\phi(\nu).\qedhere
	\end{equation*}
	\end{proof}
	
	\subsection{\textbf{Non-parametric estimation}}\label{subsec:Proof_Nonparametric}
	
	\begin{proof}[Proof of Theorem \ref{thm:nonparaCLT}]
	By Proposition \ref{prop:CLT_Localized}, we only need to treat the approximation error:
	\begin{align*}
		\left|\frac{\int_{t_0-\delta_t}^{t_0+\delta_t}\int_{\Lambda_{\delta_y}(y_0)}[B\inv F(X_t)](y)^2\theta(y,t)\,\D y\D t}{\int_{t_0-\delta_t}^{t_0+\delta_t}\int_{\Lambda_{\delta_y}(y_0)}[B\inv F(X_t)](y)^2\,\D y\D t} - \theta(y_0,t_0)\right| \hspace{-10em}&\\
		&\le\frac{\int_{t_0-\delta_t}^{t_0+\delta_t}\int_{\Lambda_{\delta_y}(y_0)}[B\inv F(X_t)](y)^2|\theta(y,t)-\theta(y_0,t_0)|\,\D y\D t}{\int_{t_0-\delta_t}^{t_0+\delta_t}\int_{\Lambda_{\delta_y}(y_0)}[B\inv F(X_t)](y)^2\,\D y\D t}\\
		&\le (2\delta_t)^{\eta_t}+\operatorname{diam}(\Lambda_{\delta_y}(y_0))^{\eta_y}
		\lesssim \delta_t^{\eta_t} + \delta_y^{\eta_y}
	\end{align*}
	for the diameter $\operatorname{diam}(\Lambda_{\delta_y}(y_0))$ of $\Lambda_{\delta_y}(y_0)$.
	\end{proof}

	\begin{proof}[Proof of Corollary \ref{cor:nonparaOptimalChoice}]
		Theorem \ref{thm:nonparaCLT} shows
		\begin{equation*}
			\hat{\theta}_{\nu,\delta_t,\delta_y}^{\operatorname{loc}}-\theta(y_0,t_0) = \mathcal{O}_{\prob*{}}\left(\delta_t^{\eta_t}+\delta_y^{\eta_y} + \delta_y^{-d/2}\delta_t^{-1/2}\phi(\nu)^{-1/2}\right).
		\end{equation*}
		Optimizing $\delta_y$ and $\delta_t$, as stated, yields the claim.
	\end{proof}
	
			\subsection{\textbf{Regularity results}}
	
	The quantification of the approximation $\hat{\theta}_{\nu,\delta_t ,\delta_y}^{\operatorname{disc}}$ to $\hat{\theta}_\nu$ depends on the spatial regularity of $X_t$ as $\nu\to 0$. We will make use of the factorization method of \cite{DaPrato2014}. For the spatial discretization we will further embed the interpolation spaces into Sobolev and subsequently into Hölder spaces. Recall	$\bar{X}_t= \int_0^t S_{\nu(t-s)}B\,\D W_s$.

\begin{lemma}\label{lem:spatial_reg_linear}
			Grant Assumption \ref{assump:discretize} with $0\le \gamma<(1-1/\beta)/2-1/p$. Then $\bar{X}_{\MTemptyplaceholder}\in C([0,T],\mathcal{B}_{\gamma,p})$ with
			\begin{equation*}
				\EV*{\norm*{\bar{X}_t}_{\gamma,p}^2}\lesssim \nu^{-2\gamma}\phi(\nu),\quad t\in [0,T].
			\end{equation*}
\end{lemma}

\begin{proof}
			For $t=0$ the claim is trivial, hence consider $t>0$. $\bar{X}_{\MTemptyplaceholder}\in C([0,T],\mathcal{B}_{\gamma,p})$ follows directly from Theorem 5.25 of \cite{DaPrato2014}. We are left with tracing the impact of $\nu\to 0$ for the norm bound through the proof. Let $\gamma+1/p<\eta< (1-1/\beta)/2$ and define $Y_\eta^\nu(t)\coloneqq\int_0^t (t-s)^{-\eta}S_{\nu(t-s)} B \,\D W_s$. Using the Malliavin divergence operator $\delta$ and $\norm*{S_u}_{\operatorname{HS}}^2\lesssim u^{-1/\beta}$ for $u>0$, we obtain
			\begin{align*}
				\EV*{Y_\eta^\nu(t,x)^2} &= \EV*{\delta((t-\MTemptyplaceholder)^{-\eta}G_{\nu(t-\MTemptyplaceholder)}(x,\MTemptyplaceholder))^2} =\norm*{(t-\MTemptyplaceholder)^{-\eta}G_{\nu(t-\MTemptyplaceholder)}(x,\MTemptyplaceholder)}_{\mathbb{H}}^2 \\
				&= \int_0^t(t-s)^{-2\eta}\norm*{B G_{\nu(t-s)}(x,\MTemptyplaceholder)}^2\,\D s\lesssim \int_0^t (t-s)^{-2\eta}\norm*{S_{\nu(t-s)}}_{\operatorname{HS}}^2 \,\D s\\
				&\lesssim \nu^{-1/\beta}\int_0^t (t-s)^{-2\eta - 1/\beta}\,\D s \lesssim \nu^{-1/\beta}t^{1-2\eta -1/\beta}.
			\end{align*}
			Finiteness of the integral follows from our choice $2\eta<1-1/\beta$. Hence, using $\EV{\abs{Y_\eta^\nu(t,x)}^{p}}\lesssim \EV{\abs{Y_\eta^\nu(t,x)}^2}^{p/2}$ by Gaussianity, we see
			\begin{equation*}
				\EV*{\norm*{Y_\eta^\nu(t)}_{0,p}^p}\lesssim \nu^{-p/(2\beta)} t^{(1-2\eta-1/\beta)p/2}.
			\end{equation*}
			From Theorem 5.10 of \cite{DaPrato2014} we know
			\begin{equation*}
				\bar{X}_t = \frac{\sin(\eta\pi)}{\pi} \int_0^t (t-s)^{\eta-1}S_{\nu(t-s)} Y_\eta^\nu(s)\,\D s.
			\end{equation*}			
			Up to the constant factor, we bound the right-hand side, using Hölder's inequality with $\frac 1p+\frac 1q=1$,
			\begin{align*}
				\EV*{\norm*{\int_0^t (t-s)^{\eta-1}S_{\nu(t-s)} Y_\eta^\nu (s)\,\D s}_{\gamma,p}^2} \hspace{-10em}&\\
				&\lesssim \EV*{\left(\int_0^t \nu^{-\gamma}(t-s)^{\eta-1-\gamma}\norm*{Y_\eta^\nu(s)}_{0,p}\,\D s\right)^2}\\
				&\le	\nu^{-2\gamma}\left(\int_0^t (t-s)^{q(\eta-1-\gamma)}\,\D s\right)^{2/q}
					\EV*{\left(\int_0^t \norm*{Y_\eta^\nu(s)}_{0,p}^p\,\D s\right)^{2/p}}\\
			&\lesssim \nu^{-2\gamma} t^{(1+q(\eta-1-\gamma))2/q} \left(\EV*{\int_0^t \norm*{Y_\eta^\nu(s)}_{0,p}^p\,\D s}\right)^{2/p}\\
			&\lesssim \nu^{-1/\beta-2\gamma}t^{1-2\gamma-1/\beta},
			\end{align*}
noting that $\eta-1-\gamma>1/p-1=-1/q$. Because of $1-2\gamma-1/\beta>0$ we can omit the $t$-dependence.
\end{proof}

In fact, the following lemma shows that the Hilbert space structure in the case $p=2$ allows for a sharper estimate.
		\begin{lemma}\label{lem:spatial_reg_linear_HS}
			Grant Assumption \ref{assump:discretize} with $p=2$ and $0\le \gamma<(1-1/\beta)/2$.  Then $\bar{X}_{\MTemptyplaceholder}\in C([0,T],\mathcal{B}_{\gamma,2})$ with
			\begin{equation*}
				\EV*{\norm*{\bar{X}_t}_{\gamma,2}^2}\lesssim \nu^{-2\gamma}\phi(\nu),\quad t\in [0,T].
			\end{equation*}
		\end{lemma}
		\begin{proof}
			Note that for any $1/(2\beta)<\rho<1/2$, $(-A)^{-\rho}$ is Hilbert Schmidt and $\gamma<1/2-\rho$. The first part of the claim follows from Theorem 5.10 of \cite{Hairer2009} and the estimate
			\begin{equation*}
				\int_0^T t^{-2\kappa}\norm*{S_{\nu t}(-A)^\gamma}_{\operatorname{HS}}^2\,\D t\le \nu^{-2\gamma+2\rho}\norm*{(-A)^\rho}_{\operatorname{HS}}^2\int_0^T t^{-2\kappa-2\gamma+2\rho}\,\D t<\infty
			\end{equation*}
			for any $\kappa\in (0,(1-2\gamma+2\rho)/2)$. The second claim follows from
			\begin{align*}
				\EV*{\norm*{\bar{X}_t}_{\gamma,2}^2} &= \int_0^t\norm*{(-A)^\gamma S_{\nu(t-s)}}_{\operatorname{HS}}^2\,\D s = \sum_{k\in\mathbb{N}}\int_0^t e^{-2\mu_k\nu (t-s)} \mu_k^{2\gamma}\,\D s \\
&\sim \sum_{k\in\mathbb{N}} \big((\mu_k\nu)^{-1}\wedge t\big) \mu_k^{2\gamma} \sim \sum_{k\in\mathbb{N}} \big((k^\beta\nu)^{-1}\wedge t\big) k^{2\gamma\beta}\lesssim \nu^{-1/\beta-2\gamma},
\end{align*}
as $(2\gamma-1)\beta<-1$. Recalling  $\phi(\nu)\sim\nu^{-1/\beta}$ completes the proof.
\end{proof}

\begin{lemma}\label{lem:spat_reg}
			Grant Assumption \ref{assump:discretize} with either $0\le \gamma<(1-1/\beta)/2-1/p$ or $p=2$ and $0\le \gamma<(1-1/\beta)/2$. Then $X_t\in \mathcal{B}_{\gamma,p}$ for all $t\in(0,T]$ with
			\begin{equation*}
				\EV*{\norm*{X_t}_{\gamma,p}^2}\lesssim \nu^{-2\gamma}(\phi(\nu) + t^{-2\gamma}).
			\end{equation*}
		\end{lemma}

		\begin{proof}
			The first part of the claim follows from Theorem 6.5 of \cite{Hairer2009}. For a quantification of the $\norm*{}_{0,p}$-norm we compute with Lemma \ref{lem:spatial_reg_linear} for $\gamma=0$
			\begin{align*}
				\EV*{\norm*{X_t}_{0,p}^2}&\le 3\norm*{S_{\nu t} X_0}_{0,p}^2 + 3t\int_0^t \EV*{\norm*{S_{\nu(t-s)} F(X_s)}_{0,p}^2}\,\D s + 3\EV*{\norm*{\bar{X}_t}_{0,p}^2}\\
				&\lesssim \EV*{\norm*{X_0}_{0,p}^2} + t\Big(\int_0^t \EV*{\norm*{X_s}_{0,p}^2}\,\D s+t\Big) + \nu^{-1/\beta}.
			\end{align*}
			Using Gronwall's inequality, we arrive at
			\begin{equation*}
				\EV*{\norm*{F(X_t)}_{0,p}^2}\lesssim \EV*{\norm*{X_t}_{0,p}^2}\lesssim \nu^{-1/\beta}.
			\end{equation*}
			Combined with Lemmas \ref{lem:spatial_reg_linear} and \ref{lem:spatial_reg_linear_HS}, this allows us to compute the interpolation norm
			\begin{align*}
				\EV*{\norm*{X_t}_{\gamma,p}^2}&\le 3\norm*{S_{\nu t} X_0}_{\gamma,p}^2 + 3t\int_0^t \EV*{\norm*{S_{\nu(t-s)} F(X_s)}_{\gamma,p}^2}\,\D s + 3\EV*{\norm*{\bar{X}_t}_{\gamma,p}^2}\\
				&\lesssim (\nu t)^{-2\gamma}\EV*{\norm*{X_0}_{0,p}^2} + t \nu^{-2\gamma}\int_0^t(t-s)^{-2\gamma} \EV*{\norm*{F(X_s)}_{0,p}^2}\,\D s + \nu^{-1/\beta-2\gamma}\\
				&\lesssim (\nu t)^{-2\gamma} + t\nu^{-2\gamma-1/\beta}t^{-2\gamma + 1}+\nu^{-1/\beta-2\gamma}\\
				&\lesssim \nu^{-2\gamma-1/\beta} + t^{-2\gamma}\nu^{-2\gamma}.
			\end{align*}
			Note that the time integral is finite because of $\gamma<(1-1/\beta)/2-1/p<1/2$.
		\end{proof}

\begin{corollary}\label{cor:Höldernorm}
			Grant Assumption \ref{assump:discretize} with either $0\le \gamma<(1-1/\beta)/2-1/p$ and $\tilde{\gamma}-d/p>0 $ or $p=2$,$ \tilde{\gamma}>d/2$ and $0<\gamma<(1-1/\beta)/2$. Then $X_t\in C^{\tilde{\gamma}-d/p}$ holds for all $t\in (0,T]$ and
			\begin{equation*} \max\left(\EV*{\norm*{F(X_t)}_{C^{\tilde{\gamma}-d/p}}^2},\EV*{\norm*{X_t}_{C^{\tilde{\gamma}-d/p}}^2}\right)\lesssim \nu^{-2\gamma}(\phi(\nu)+t^{-2\gamma}).
			\end{equation*}
		\end{corollary}
		\begin{proof}
			Follows from Assumption \ref{assump:discretize}, Lemma \ref{lem:spat_reg} and the Sobolev Embedding Theorem (Theorem 8.2 of \cite{DiNezza2012}), which yield the continuous embeddings	$\mathcal{B}_{\gamma,p}\hookrightarrow W^{\tilde{\gamma},p}\hookrightarrow C^{\tilde{\gamma}-1/p}$.
\end{proof}

\subsection{\textbf{Temporal discretization}}
		
		First, we control the deviation of the time-discretised estimator $\hat{\theta}_{\nu,\delta_t }^{\operatorname{disc}}$ from $\hat{\theta}_\nu$. We will make use of the regularity results of Lemma \ref{lem:spat_reg} in the case $p=2$. In order to be consistent with the previous chapters, we still employ the notation $\norm*{}=\norm*{}_{0,2}$.
		\begin{lemma}\label{lem:temporal_Hölder}
			Grant Assumption \ref{assump:discretize} with $p=2$, $\gamma< (1-1/\beta)/2$. For $t\in[t_k,t_{k+1}]$, $k\in\Set*{1,\dots,n}$, we have
			\begin{equation*}
				\EV*{\norm*{X_t-X_{t_k}}^2}\lesssim \delta_{t,k} ^{2\gamma}\left(t_k^{-2\gamma} +\phi(\nu)\right)\le 1 +\delta_{t,k}^{2\gamma}\phi(\nu).
			\end{equation*}
		\end{lemma}
		\begin{proof}
		Recall that there exists a constant $C$, only depending on $A$ and $\gamma$, such that $\norm*{S_tz-z}\le Ct^\gamma\norm*{z}_\gamma$ holds for $t\ge 0$. Using Lemma \ref{lem:spat_reg} with $p=2$ and \eqref{eq:growthbound_L^2}, we bound
		\begin{align*}
			 \EV*{\norm*{X_t-X_{t_k}}^2}
			&\le
			\begin{multlined}[t]
				3\EV*{\norm*{S_{\nu(t-t_k)} X_{t_k}-X_{t_k}}^2} + 3\int_{t_k}^t\EV*{\norm*{S_{\nu(t-s)} F(X_s)}^2}\,\D s\\
				 +3\EV*{\norm*{\int_{t_k}^t S_{\nu(t-s)} B \,\D W_s}^2}
			\end{multlined}\\
			&\le\begin{multlined}[t]
				3C(\nu\delta_{t,k} )^{2\gamma}\EV*{\norm*{X_{t_k}}_{\gamma,2}^2} +3\int_{t_k}^t\EV*{\norm*{F(X_s)}^2}\,\D s \\
				  +3 \int_{t_k}^t \norm*{S_{\nu(t-s)}B}_{\operatorname{HS}}^2 \,\D s
				\end{multlined}\\
			&\lesssim \delta_{t,k}^{2\gamma}t_k^{-2\gamma} +\delta_{t,k}^{2\gamma}\nu^{-1/\beta} +  \delta_{t,k}^{1-1/\beta} \nu^{-1/\beta} + \delta_{t,k} + \delta_{t,k} \nu^{-1/\beta}\\
			&\lesssim \delta_{t,k} ^{2\gamma}\left(t_k^{-2\gamma} +\nu^{-1/\beta}\right).
		\end{align*}
		In the last line we used  $\gamma<(1-1/\beta)/2$ and $\beta>1$ to identify the dominating term for $\nu,\delta_t \lesssim 1$. Noting that $\delta_{t,k} t_k^{-1}\gtrsim 1$, we conclude the proof.
		\end{proof}
		\begin{proposition}\label{prop:discret_temp}
			Grant Assumption \ref{assump:discretize} with $p=2$, $\gamma<(1-1/\beta)/2$, assume  the uniform bound $\norm*{B^{-2} F(z)}_{\gamma,2}\lesssim \norm*{z}_{\gamma,2}+1$, $z\in \mathcal{B}_{\gamma,2}$, and
 $\delta_t =\smallo(\nu^{1/(2\beta\gamma)})$. Then
			\begin{equation*}
				\EV*{\abs*{\hat{\theta}_\nu-\hat{\theta}_{\nu,\delta_t }^{\operatorname{disc}}}}=\smallo(\phi(\nu)^{-1/2}).
			\end{equation*}
		\end{proposition}
		\begin{proof}
		Note that we have under $\prob*[\theta]{}$ the representation
		\begin{equation*}
			X_{t_{k+1}} = S_{\nu\delta_{t,k}} X_{t_k} + \theta\int_{t_k}^{t_{k+1}} S_{\nu(t_{k+1}-s)} F(X_s)\,\D s +  \int_{t_k}^{t_{k+1}}S_{\nu(t_{k+1}-s)} B\,\D W_s,
		\end{equation*}
		which implies
		\begin{equation}
			\hat{\theta}_{\nu,\delta_t }^{\operatorname{disc}} =\frac{\sum_{k=1}^{n-1}\iprod*{B\inv F(X_{t_k})}{B\inv \left[\theta \int_{t_k}^{t_{k+1}}S_{\nu(t_{k+1}-s)} F(X_s)\,\D s + \int_{t_k}^{t_{k+1}}S_{\nu(t_{k+1}-s)} B\,\D W_s\right]}}{\sum_{k=1}^{n-1} \delta_{t,k}\norm*{B\inv F(X_{t_k})}^2}.\label{eq:esti_temp_decomp}
		\end{equation}
		This decomposition allows for the representation
		\begin{align*}
			\hat{\theta}_{\nu,\delta_t }^{\operatorname{disc}} &=
			\begin{multlined}[t]
			\frac{\sum_{k=1}^{n-1} \left[\delta_{t,k} \theta \norm*{B\inv F(X_{t_k})}^2 + \int_{t_k}^{t_{k+1}}\iprod*{B S_{\nu(t_{k+1}-s)} B\inv B\inv F(X_{t_k})}{\D W_s}\right]}{\sum_{k=1}^{n-1} \delta_{t,k}\norm*{B\inv F(X_{t_k})}^2} +\\
			\frac{\theta \sum_{k=1}^{n-1}\int_{t_k}^{t_{k+1}}\iprod*{B\inv F(X_{t_k})}{B\inv S_{\nu(t_{k+1}-s)} F(X_s)-B\inv F(X_{t_k})}\,\D s}{\sum_{k=1}^n \delta_{t,k}\norm*{B\inv F(X_{t_k})}^2}
			\end{multlined}\\
			&=\colon\theta + \frac{\hat{\mathcal{M}}_{\nu,\delta_t }+\mathcal{R}_{\nu,\delta_t }}{\hat{\mathcal{I}}_{\nu,\delta_t }}.
		\end{align*}
		For $\delta_t =\smallo(\nu^{1/(2\beta\gamma)})$ we will establish the bounds
		\begin{align*}
			\abs*{\hat{\mathcal{M}}_{\nu,\delta_t }-\mathcal{M}_{\nu}}&\le
			\begin{multlined}[t]
				\left|\sum_{k=1}^{n-1}\int_{t_k}^{t_{k+1}}\iprod*{B S_{\nu(t_{k+1}-s)} B\inv B\inv F(X_{t_k})}{\D W_s} -\right.\\
				\left. \int_{t_1}^T \iprod*{B\inv F(X_t)}{\D W_t}\right| +\abs*{\int_0^{t_1}  \iprod*{B\inv F(X_t)}{\D W_t}}
				\end{multlined}\\
				&=\colon\mathcal{M}_1+ \mathcal{M}_2\overset{!}{=}\mathcal{O}_{\prob*{}}(1),\\
			\abs*{\hat{\mathcal{I}}_{\nu,\delta_t }-\mathcal{I}_\nu}&\le \abs*{\sum_{k=1}^{n-1}\delta_{t,k}\norm*{B\inv F(X_{t_k})}^2-\int_{t_1}^T\norm*{B\inv F(X_t)}^2\,\D t} + \int_0^{t_1} \norm*{B\inv F(X_t)}^2\,\D t\\
			&=\colon \mathcal{I}_1+\mathcal{I}_2\overset{!}{=}\mathcal{O}_{\prob*{}}(\nu^{-1/(2\beta)}),\\
			\abs*{\mathcal{R}_{\nu,\delta_t }}&\le  \theta \sum_{k=1}^{n-1}\int_{t_k}^{t_{k+1}}|\iprod*{B\inv F(X_{t_k})}{B\inv S_{\nu(t_{k+1}-s)} F(X_s)-B\inv F(X_{t_k})}|\,\D s\\
			&\overset{!}{=} \smallo_{\prob*{}}(\nu^{-1/(2\beta)}).
		\end{align*}
		This suffices, as we need
		\begin{align*}
			\sqrt{\phi(\nu)}\abs*{\hat{\theta}_{\nu,\delta_t }^{\operatorname{disc}} -\hat{\theta}_\nu}&\lesssim \nu^{-1/(2\beta)}\abs*{\frac{\hat{\mathcal{M}}_{\nu,\delta_t }+\mathcal{R}_{\nu,\delta_t }}{\hat{\mathcal{I}}_{\nu,\delta_t }}-\frac{\mathcal{M}_\nu}{\mathcal{I}_\nu}}\\
			&\le \frac{\frac{\abs*{\hat{\mathcal{M}}_{\nu,\delta_t }-\mathcal{M}_\nu}+\abs*{\mathcal{R}_{\nu,\delta_t }}}{\nu^{-1/(2\beta)}}}{\frac{\mathcal{I}_\nu}{\nu^{-1/\beta}}-  \frac{\abs*{\hat{\mathcal{I}}_{\nu,\delta_t }-\mathcal{I}_\nu}}{\nu^{-1/\beta}}} + \nu^{-1/(2\beta)}\frac{\mathcal{M}_\nu}{\mathcal{I}_\nu}\frac{\frac{\abs*{\hat{\mathcal{I}}_{\nu,\delta_t }-\mathcal{I}_\nu}}{\nu^{-1/\beta}}}{\frac{\mathcal{I}_\nu}{\nu^{-1/\beta}}-  \frac{\abs*{\hat{\mathcal{I}}_{\nu,\delta_t }-\mathcal{I}_\nu}}{\nu^{-1/\beta}}}\\
			&\overset{!}{=}\smallo_{\prob*{}}(1)
		\end{align*}
		and $\nu^{-1/(2\beta)}\abs*{\mathcal{M}_\nu}/\mathcal{I}_\nu=\mathcal{O}_{\prob*{}}(1)$, $\mathcal{I}_\nu^{-1}=\mathcal{O}_{\prob*{}}(\nu^{1/\beta})$ from Proposition \ref{thm:non-linearrate}. In particular, the choice of $\delta_t=\smallo(\nu^{1/(2\beta\gamma)})$ is required for the bound on $\mathcal{R}_{\nu,\delta_t}$. Taking care of the first time increment $[t_0,t_1]$ individually is needed because $X_{t_1}$ might not have sufficient regularity, as can be seen in Lemmas \ref{lem:spat_reg} and \ref{lem:temporal_Hölder}. If $X_0\in \mathcal{B}_{\gamma,2}$, this separation is not necessary.
		
		\textbf{Step 1:} We start by controlling the difference of the stochastic integrals using Itô's isometry:
		\begin{align*}
			&\EV*{\mathcal{M}_1^2}= \sum_{k=1}^{n-1}\int_{t_k}^{t_{k+1}} \EV*{\norm*{B S_{\nu(t_{k+1}-t)} B^{-2} F(X_{t_k})-B\inv F(X_t)}^2}\,\D t\\
			&\le \norm*{B}^2\sum_{k=1}^{n-1}\int_{t_k}^{t_{k+1}} \EV*{\norm*{S_{\nu(t_{k+1}-t)} B^{-2} F(X_{t_k})-B^{-2} F(X_t)}^2}\,\D t\\
			&\le
			\begin{multlined}[t]
			2\norm*{B}^2 \sum_{k=1}^{n-1}\int_{t_k}^{t_{k+1}} \EV*{\norm*{S_{\nu(t_{k+1}-t)} B^{-2} F(X_{t_k})-S_{\nu(t_{k+1}-t)} B^{-2} F(X_t)}^2} +\\
			 \EV*{\norm*{S_{\nu(t_{k+1}-t)} B^{-2} F(X_t)- B^{-2} F(X_t)}^2}\,\D t
			\end{multlined}\\
			&\le
			\begin{multlined}[t]
				2\norm*{B}^2 \sum_{k=1}^{n-1}\int_{t_k}^{t_{k+1}} \EV*{\norm*{B\inv}^4\norm*{F(X_{t_k})-F(X_t)}^2} +
				 (\nu \delta_{t,k} )^{2\gamma}\EV*{\norm*{B^{-2} F(X_t)}_{\gamma,2}^2}\,\D t
				\end{multlined}\\
			&\lesssim 1+ \delta_t^{2\gamma} \nu^{-1/\beta} + (\nu\delta_t )^{2\gamma} \int_0^T \big(t^{-2\gamma}\nu^{-2\gamma} + \nu^{-2\gamma-1/\beta}\big)\,\D t\lesssim 1+\delta_t ^{2\gamma}\nu^{-1/\beta}
			\lesssim 1,
		\end{align*}
		where we used  $\norm*{S_tz -z}\lesssim t^\gamma \norm*{z}_{\gamma,2}$  and $\delta_{t,k} \nu\le 1$. As we did not assume that $X_0\in \mathcal{B}_{\gamma,2}$, the last step requires treating the first time increment $[0,t_1]$ separately. Furthermore, using Lemma \ref{lem:non-linearexpectation}, \eqref{eq:growthbound_L^2}, $\gamma<(1-1/\beta)/2$ and $\delta_{t,0}\le C\delta_t$, we estimate
		\begin{align*}
			\EV*{\mathcal{M}_2^2}&\le \int_0^{t_1} \EV*{\norm*{B\inv F(X_t)}^2}\,\D t\lesssim \delta_{t,0}^{1-1/\beta}\nu^{-1/\beta} + \delta_{t,0}=\delta_{t,0}^{1-1/\beta}\nu^{-1/\beta} + \delta_{t,0}\\
			&= \smallo(\delta_t^{2\gamma})\nu^{-1/\beta}.
		\end{align*}
		\textbf{Step 2:} For the remainder $\mathcal{R}_{\nu,\delta_t }$ we find, again assuming $\delta_{t,k} \nu\le 1$, that
		\begin{align*}
			\frac{\EV*{\abs*{\mathcal{R}_{\nu,\delta_t }}}}{\theta} &\le  \EV*{\sum_{k=1}^{n-1}\int_{t_k}^{t_{k+1}}\abs*{\iprod*{B\inv F(X_{t_k})}{B\inv S_{\nu(t_{k+1}-t)} F(X_t)-B\inv F(X_{t_k})}}\,\D t}\\
			&\le \sum_{k=1}^{n-1} \int_{t_k}^{t_{k+1}}\EV*{\norm*{B\inv F(X_{t_k})}\norm*{B\inv S_{\nu(t_{k+1}-t)} F(X_t)-B\inv F(X_{t_k})}}\,\D t\\
			&\lesssim \begin{multlined}[t]
				\left( \sum_{k=1}^{n-1} \int_{t_k}^{t_{k+1}}\EV*{\norm*{B\inv F(X_{t_k})}^2}\,\D t \cdot\right.\\
				\left.\sum_{k=1}^{n-1}\int_{t_k}^{t_{k+1}}\left(\EV*{\norm*{S_{\nu(t_{k+1}-t)} F(X_t)-S_{\nu(t_{k+1}-t)} F(X_{t_k})}^2} +\right.\right.\\
				\left.\left.\EV*{\norm*{S_{\nu(t_{k+1}-t)} F(X_{t_k})-F(X_{t_k})}^2}\right)\,\D t\right)^{1/2}
			\end{multlined}\\
			&\lesssim
			\begin{multlined}[t]
				\left([\nu^{-1/\beta}T^{2-1/\beta} + T]\left(\sum_{k=1}^{n-1} \int_{t_k}^{t_{k+1}} \EV*{\norm*{F(X_t)-F(X_{t_k})}^2}\right.\right.+\\
				\left.\left. (\nu\delta_{t,k} )^{2\gamma}\EV*{\norm*{F(X_{t_k})}_{\gamma,2}^2}\,\D t\right)\right)^{1/2}
				\end{multlined}\\
				 &\lesssim \nu^{-1/(2\beta)}\left(1+\delta_t ^{2\gamma}\nu^{-1/\beta} + 1 + \delta_t ^{2\gamma}\nu^{-1/\beta}\right)^{1/2}\lesssim \nu^{-1/(2\beta)} + 1.
		\end{align*}
		
		\textbf{Step 3:} Finally, we control the difference of the observed Fisher information. We find
		\begin{align*}
			\EV*{\mathcal{I}_1}&\le \sum_{k=1}^{n-1}\int_{t_k}^{t_{k+1}}\EV*{\abs*{\norm*{B\inv F(X_t)}^2 - \norm*{B\inv F(X_{t_k})}^2}}\,\D t\\
			&= \sum_{k=1}^{n-1}\int_{t_k}^{t_{k+1}}\EV*{\abs*{\norm*{B\inv F(X_t)} + \norm*{B\inv F(X_{t_k})}}\abs*{\norm*{B\inv F(X_t)}-\norm*{B\inv F(X_{t_k})}}}\,\D t\\
			&\le \sqrt{2} \norm*{B\inv}^2 \sum_{k=1}^{n-1}\int_{t_k}^{t_{k+1}}\sqrt{\EV*{\norm*{F(X_t)}^2+\norm*{ F(X_{t_k})}^2}}\sqrt{\EV*{\norm*{F(X_t)-F(X_{t_k})}^2}}\,\D t\\
			&\lesssim T^{1/2-1/(2\beta)}\nu^{-1/(2\beta)}\sum_{k=1}^{n-1} \int_{t_k}^{t_{k+1}} \sqrt{1 + \delta_{t,k} ^{2\gamma}\nu^{-1/\beta}}\,\D t\lesssim\nu^{-1/(2\beta)} + \delta_t^{\gamma}\nu^{-1/\beta}\\
			&\lesssim \nu^{-1/(2\beta)} + 1.
		\end{align*}
		Additionally, we find that
		\begin{align*}
			\EV*{\mathcal{I}_2}&=\EV*{\int_0^{t_1}  \norm*{B\inv F(X_t)}^2\,\D t}\lesssim\delta_{t,0}^{1-1/\beta}\nu^{-1/\beta} + \delta_{t,0}=\delta_{t,0}^{1-1/\beta}\nu^{-1/\beta} + \delta_{t,0}\\
			&= \smallo(\delta_t^{2\gamma})\nu^{-1/\beta}.
		\end{align*}
		Recalling  $\delta_t =\smallo(\nu^{1/(2\beta\gamma)})$ concludes the proof.
		\end{proof}
		
	\subsection{\textbf{Spatial discretization}}
	
		The main idea of proof is to use Corollary \ref{cor:Höldernorm}, controlling the spatial Hölder-norm of $X_t$, and then use approximation of Hölder-continuous functions. Recall that we have $\delta_{t,k}\sim\delta_t$ (with constants independent of $\nu$) for $k=1,\dots,n$.

		\begin{proposition}\label{prop:discret_spat}
			 Grant Assumption \ref{assump:discretize} with $p>1/\tilde{\gamma}$, $\gamma<(1-1/\beta)/2-1/p$, and assume  $\delta_y = \smallo(\nu^{(\gamma+1/(2\beta))/(\tilde{\gamma}-1/p)}\delta_t^{1/(\tilde{\gamma}-1/p)})$. Then we have
			 \begin{equation*}
			 	\EV*{\abs*{\hat{\theta}_{\nu,\delta_t }^{\operatorname{disc}} -\hat{\theta}_{\nu,\delta_t ,\delta_y}^{\operatorname{disc}}}}=\smallo(\phi(\nu)^{-1/2}).
			 \end{equation*}
		\end{proposition}
		\begin{remark}
			We bound the individual deviations of the terms appearing in $\hat{\theta}_{\nu,\delta_t ,\delta_y}^{\operatorname{disc}}$ and $\hat{\theta}_{\nu,\delta_t }^{\operatorname{disc}}$. It might be possible to profit from cancellations in the denominator, similar to \eqref{eq:esti_temp_decomp}, but since the additional quantity $(L^{\delta_y} S_{\nu\delta_{t,k}} - S_{\nu\delta_{t,k}}L^{\delta_y})X_{t_k}$ is difficult to control,  we were not able to establish better spatial approximation rates.
		\end{remark}
		\begin{proof}[Proof of Proposition \ref{prop:discret_spat}]~
		
		\textbf{Step 1 (Preliminaries):} By Assumption \ref{assump:discretize}, we have for any $t\in [0,T]$ the interpolation error bound
		\begin{equation}
			\norm*{(\identity-L^{\delta_y})X_t}\le \delta_y^{\tilde{\gamma}-1/p}\norm*{X_t}_{C^{\tilde{\gamma}-1/p}}\lesssim \delta_y^{\tilde{\gamma}-1/p}\norm*{X_t}_{\gamma,p}.\label{eq:Discret_Spat_Hölder_error}
		\end{equation}
		Combining the bound \eqref{eq:Discret_Spat_Hölder_error} with Corollary \ref{cor:Höldernorm} and $\delta_y^{\tilde{\gamma}-1/p}=\smallo(\nu^{\gamma+1/(2\beta)}\delta_t)$ for our choice of $\delta_y$, we find for $t\ge t_1\ge C \nu^{1/(2\beta\gamma)}$ the estimate
		\begin{align*}
			\EV*{\norm*{L^{\delta_y} X_t - X_t}^2}&\lesssim \delta_y^{2\tilde{\gamma}-2/p}(t^{-2\gamma}\nu^{-2\gamma}+\nu^{-2\gamma-1/\beta})=\smallo(\delta_t^2).
		\end{align*}
		
		\textbf{Step 2:} Recalling the full discretization given by \eqref{eq:Discretization_full}, we define  the quantities
		\begin{align*}
			\hat{J}_{k}&\coloneqq \norm*{B\inv F(X_{t_k})}^2,\quad \tilde{J}_k \coloneqq \norm*{B\inv F(L^{\delta_y} X_{t_k})}^2\\
			\hat{R}_k&\coloneqq \iprod*{B\inv F(X_{t_k})}{B\inv X_{t_{k+1}}},\quad \tilde{R}_k\coloneqq \iprod*{B\inv F(L^{\delta_y} X_{t_k})}{B\inv L^{\delta_y} X_{t_{k+1}}}\\
			\hat{K}_k&\coloneqq \iprod*{B\inv F(X_{t_k})}{B\inv S_{\nu\delta_{t,k}} X_{t_k}}, \quad \tilde{K}_k\coloneqq \iprod*{B\inv F(L^{\delta_y} X_{t_k})}{B\inv S_{\nu\delta_{t,k}} L^{\delta_y} X_{t_k}}.
		\end{align*}
		It suffices to prove that the corresponding approximation differences satisfy
		\begin{equation}
		\sum_{k=1}^n\EV*{\abs*{\hat{A}_k-\tilde{A}_k}}=\smallo(\nu^{-1/(2\beta)}),\quad A=R,K,\quad \sum_{k=1}^n\delta_{t,k}\EV*{\abs*{\hat{J}_k-\tilde{J}_k}}=\smallo(\nu^{-1/\beta}),\label{eq:Claim_Spatial}
		\end{equation}
		because
		\begin{align*}
			\nu^{-1/(2\beta)}\abs*{\hat{\theta}_{\nu,\delta_t }^{\operatorname{disc}}  - \hat{\theta}_{\nu,\delta_t ,\delta_y}^{\operatorname{disc}}}&\le
			\begin{multlined}[t]
				\nu^{-1/(2\beta)}\frac{\sum_{k=1}^{n-1} [\abs{\hat{T}_{t_k} -\tilde{T}_{t_k}} + \abs{\hat{R}_{t_k}-\tilde{R}_{t_k}}]}{\sum_{k=1}^{n-1}\delta_{t,k} \tilde{J}_{t_k}} \\
			 +\nu^{-1/(2\beta)}\frac{\sum_{k,l=1}^{n-1}\abs{\hat{T}_{t_k} + \hat{R}_{t_k}}(t_{l+1}-t_l)\abs{\hat{J}_{t_l}-\tilde{J}_{t_l}}}{\sum_{k,l=1}^{n-1}\delta_{t,k}\delta_{t,l}\hat{J}_{t_k}\tilde{J}_{t_l}}
			 \end{multlined}\\
			&\le
			\begin{multlined}[t]
				\frac{\nu^{1/(2\beta)}  \sum_{k=1}^{n-1} [\abs{\hat{T}_{t_k} -\tilde{T}_{t_k}} + \abs{\hat{R}_{t_k}-\tilde{R}_{t_k}}]}{\nu^{1/\beta}\left(\sum_{k=1}^{n-1}\delta_{t,k} \left(\hat{J}_{t_k}-\abs{\hat{J}_{k_k}-\tilde{J}_{t_k}}\right]\right)} \\
				 +\nu^{-1/(2\beta)}\abs{\hat{\theta}_{\nu,\delta_t }^{\operatorname{disc}}}\frac{\nu^{1/\beta} \sum_{k=1}^{n-1} \delta_{t,k}\abs{\hat{J}_{t_k}-\tilde{J}_{t_k}}}{\nu^{1/\beta}\left(\sum_{k=1}^{n-1} \delta_{t,k} \left[\hat{J}_{t_k} - \abs{\hat{J}_{t_k}-\tilde{J}_{t_k}}\right]\right)}
			\end{multlined}
		\end{align*}
		and recalling the stochastic boundedness of $\nu^{-1/(2\beta)}\hat{\theta}_{\nu,\delta_t }^{\operatorname{disc}} $ as well as $(\sum_{k=1}^{n-1}\delta_{t,k}\hat{J}_k)^{-1}=\mathcal{O}_{\prob*{}}(\nu^{1/\beta})$ from Proposition \ref{prop:discret_temp}.

	From Lemma \ref{lem:spat_reg} it follows that for $k=1,\dots, n-1$
		\begin{equation*}
			\max\left(\EV*{\norm*{B\inv F(X_{t_k})}^2},\EV*{\norm*{B\inv X_{t_{k+1}})}^2},\EV*{\norm*{B\inv S_{\nu\delta_{t,k}}(X_{t_k})}^2}\right)\lesssim \nu^{-1/\beta}.
		\end{equation*}
		and
		\begin{align*}
			\EV*{\norm*{B\inv F(X_{t_k}) - B\inv F(L^{\delta_y}X_{t_k})}^2}&\le \norm*{B\inv}^2M^2 \EV*{\norm*{X_{t_k}-L^{\delta_y}X_{t_k}}^2}=\smallo(\delta_t^2),\\
			\EV*{\norm*{B\inv X_{t_{k+1}}-B\inv L^{\delta_y} X_{t_{k+1}}}^2}&=\smallo(\delta_t^2),\\
			\EV*{\norm*{B\inv S_{\nu \delta_{t,k}}X_{t_k}-B\inv S_{\nu \delta_{t,k}}L^{\delta_y}X_{t_k}}^2}&\le \norm*{B\inv}^2\EV*{\norm*{X_{t_k}-L^{\delta_y}X_{t_k}}^2}=\smallo(\delta_t^2).
		\end{align*}
	The general inequality $\abs*{\iprod*{a}{b}-\iprod*{a'}{b'}}\le \norm*{a}\norm*{b-b'} + \norm*{a-a'}\left(\norm*{b} + \norm*{b'-b}\right)$, $a,b, a',b'\in \mathcal{H}$,
		yields
		\begin{equation*} \max\left(\EV*{\abs*{\hat{J}_k-\tilde{J}_k}},\EV*{\abs*{\hat{R}_k-\tilde{R}_k}},\EV*{\abs*{\hat{K}_k-\tilde{K}_k}}\right)=\smallo(\delta_{t,k}\nu^{-1/(2\beta)})
		\end{equation*}
		for all $k=1,\dots, n-1$. This implies \eqref{eq:Claim_Spatial} and completes the proof.
		\end{proof}

	\section{Technical tool}\label{sec:TechnicalTool}
	
			\begin{proposition}\label{prop:LowerboundNorm}
			Let $(V,\iprod*{}{})$ be an inner product space. Then we have for any $a,b\in V$ and $\alpha\in(0,1)$ that
			\begin{equation*}
				\norm*{a+b}^2\ge \alpha \norm*{a}^2 - \frac{\alpha}{1-\alpha}\norm*{b}^2.
			\end{equation*}
		\end{proposition}
		\begin{proof}
After division by $\alpha$ this follows directly from Young's inequality $\norm{v+w}^2\le (1+t)\norm{v}^2+(1+t^{-1})\norm{w}^2$ with $v=a+b$, $w=-b$, $t=\alpha^{-1}-1$.
\end{proof}

\end{appendix}

\bibliographystyle{imsart-number} 
\bibliography{literature}

\end{document}